\documentclass{amsart}

\usepackage{amsmath,amssymb,amsthm,amsrefs}
\usepackage{enumitem}
\usepackage{tikz-cd}
\usepackage{tensor}

\newcommand{\Q}{\mathbb{Q}}
\newcommand{\C}{\mathbb{C}}
\newcommand{\N}{\mathbb{N}}
\newcommand{\PP}{\mathbb{P}}
\newcommand{\Z}{\mathbb{Z}}
\newcommand{\BB}{\mathcal{B}}
\newcommand{\DD}{\mathcal{D}}
\newcommand{\cS}{\mathcal{S}}
\newcommand{\A}{\mathcal{A}}
\newcommand{\TT}{\mathcal{T}}
\newcommand{\QQ}{\mathcal{Q}}
\newcommand{\m}{\mathfrak{m}}
\newcommand{\bbS}{\mathbb S}
\newcommand{\cL}{\mathcal{L}}
\newcommand{\cO}{\mathcal{O}}
\newcommand{\cR}{\mathcal{R}}

\newcommand{\Gr}{\mathsf{Gr}}
\newcommand{\gr}{\mathsf{gr}}
\newcommand{\qgr}{\mathsf{qgr}}
\newcommand{\bimod}[3]{\tensor[_{#1}]{#2}{_{#3}}}
\newcommand{\ld}[1]{\tensor[^{\ast}]{#1}{}}
\newcommand{\rd}[1]{\tensor[]{#1}{^{\ast}}}
\newcommand{\coh}{\mathsf{coh}}
\newcommand{\fl}{\mathsf{fl}}
\newcommand{\Ab}{\mathsf{Ab}}
\newcommand{\diag}[1]{{#1}_{\Delta}}
\newcommand{\del}{\partial}

\DeclareMathOperator{\Aut}{Aut}
\DeclareMathOperator{\Hom}{Hom}
\DeclareMathOperator{\coker}{coker}

\newtheorem{theorem}{Theorem}[section]
\newtheorem{lemma}[theorem]{Lemma}
\newtheorem{propn}[theorem]{Proposition}
\newtheorem{cor}[theorem]{Corollary}

\theoremstyle{definition}
\newtheorem{defn}[theorem]{Definition}
\newtheorem{ex}[theorem]{Example}
\newtheorem{remark}[theorem]{Remark}
\numberwithin{equation}{section}

\begin{document}

\title[Complements of pt. schemes of noncommutative projective lines]{Complements of the point schemes of noncommutative projective lines}

\author{Jackson Ryder}

\address{Department of Mathematics and Statistics, University of New South Wales}
\email{jackson.ryder@unsw.edu.au}

\thanks{This work was supported by an Australian Government Research Training Program (RTP) Scholarship.}

\begin{abstract}
    Recently, Chan and Nyman constructed noncommutative projective lines via a noncommutative symmetric algebra for a bimodule $V$ over a pair of fields. These noncommutative projective lines contain a canonical closed subscheme (the point scheme) determined by a normal family of elements in the noncommutative symmetric algebra. We study the complement of this subscheme when $V$ is simple, the coordinate ring of which is obtained by inverting said normal family. We show that this localised ring is a noncommutative Dedekind domain of Gelfand-Kirillov dimension 1. Furthermore, the question of simplicity of these Dedekind domains is answered by a similar dichotomy for an analogous open subscheme of the noncommutative quadrics of Artin, Tate and Van den Bergh.
\end{abstract}

\maketitle

\section{Introduction}
The projective line $\PP^1(k)$ over a field $k$ may be defined in a coordinate-free manner as the projectivisation of the symmetric algebra $\mathrm{Proj}(\bbS(W))$ of a 2-dimensional $k$-vector space $W$. In \cite{chan_species_2016, chan_representation_2019}, Chan and Nyman construct noncommutative projective lines via an analogous construction for \textit{two-sided vector spaces}: bimodules $V$ over a pair of division rings $D_0$ and $D_1$. An analogue of the symmetric algebra of $V$ is constructed via Van den Bergh's noncommutative symmetric algebra \cite{van_den_bergh_non-commutative_2012}. When $D_0=K_0$ and $D_1=K_1$ are fields and $V$ has dimension 2 on both sides (as considered in \cite{chan_species_2016}), the resulting $\Z$-indexed algebra $\bbS^{nc}(V)$ has similar properties to a polynomial ring $k[x,y]$. The noncommutative projective line $\PP^{nc}(V)$ is then constructed as the quotient of the category of graded $\bbS^{nc}(V)$-modules by the subcategory $\mathsf{Tors}(\bbS^{nc}(V))$ of torsion modules, considered as a noncommutative projective scheme in the sense of Artin and Zhang \cite{artin_noncommutative_1994}. The analogy with the commutative projective line extends further than the construction of $\PP^{nc}(V)$, as $\PP^{nc}(V)$ shares a number of properties with the usual projective line. Both $\bbS^{nc}(V)$ and $\PP^{nc}(V)$ are noetherian, and $\PP^{nc}(V)$ has a tilting sheaf \cite[Proposition 7.2]{chan_species_2016} and satisfies analogues of Serre finiteness/vanishing \cite[Theorem 7.3]{chan_species_2016}, Serre duality \cite[Theorem 9.1]{chan_species_2016} and Grothendieck splitting \cite[Proposition 10.2]{chan_species_2016}. 

A unique feature of these noncommutative projective lines $\PP^{nc}(V)$ is the existence of a canonical closed subscheme $g=0$ of $\PP^{nc}(V)$ (the \textit{point scheme}) defined by a \textit{normal family of elements}: a family of elements $g=\{g_i \in \bbS^{nc}(V)_{i,i+2}\}_{i \in \Z}$ of $\bbS^{nc}(V)$ such that $g\bbS^{nc}(V) = \bbS^{nc}(V)g$ (see Definition 2.4 for the precise definition). This family determines a two-sided ideal $(g)$ of the noncommutative symmetric algebra, the quotient by which we view as the coordinate ring of a closed subscheme of $\PP^{nc}(V)$. The geometric notion of point schemes from \cite{van_den_bergh_non-commutative_2012} is not available in this setting, however $\bbS^{nc}(V)/(g)$ is analogous to the `projectively commutative' quotient defined by the point scheme. As shown in \cite[Section 5]{chan_species_2016}, the quotient $\bbS^{nc}(V)/(g)$ is a twisted ring, hence is noetherian, and a Hilbert basis argument is used to show $\bbS^{nc}(V)$ is also noetherian.

If we let $g_{i}^n := g_i g_{i+2} \ldots g_{i+2n}$ for $n \geq 0$, then by the definition of normality of the family $g$ the set $\{ g_{i}^n \mid i \in \Z, n \geq 0 \}$ (the set of products of the elements $g_i$, viewed as the set of `powers' of $g$) satisfies the left and right Ore conditions. This allows us to formally invert the family $g$ and obtain the $\Z$-indexed algebra $\Lambda := \bbS^{nc}(V)[g^{-1}]$. The `degree 0' component $\Lambda_{00}$ of $\Lambda$ is the coordinate ring of the `affine open' complement of the point scheme. For non-simple $V$, there is a field $k$, a $k$-automorphism $\psi$ and a $\psi$-derivation $\delta$ such that $\Lambda_{00}$ is isomorphic to a skew polynomial ring $k[t;\psi, \delta]$ \cite[Proposition 10.16]{chan_species_2016}. For simple two-sided vector spaces, however, no such description is known.

We are mostly interested in the case where $V$ is simple, so that $V = \bimod{K_0}{F}{K_1}$ for a field $F$ and index 2 subfields $K_0,K_1$. Section 2 contains the necessary background for the rest of the paper. In the third section we construct decompositions of the tensor and noncommutative symmetric algebras (c.f. Definitions 2.2 and 2.3) into direct sums of irreducible subbimodules of the form $F_{\rho}$, with underlying abelian group $F$ and right multiplication twisted by $\rho \in \Aut(F)$. The two Galois involutions $\tau_0$ and $\tau_1$ generate a subgroup $H$ of $\mathrm{Aut}(F)$, and we let $G_i = \mathrm{Gal}(F/K_{i'})$ denote the Galois groups, where $i' = i ~(\mathrm{mod}~2)$. We will abuse notation and implicitly identify our integer subscripts with their residue modulo 2. We then have a map $\overline{(-)}$ from the alternating product of Galois groups $G_{ij} = G_{i+1} \times G_{i+2} \times \ldots \times G_{j-1}$ to $H$, sending $(s_{i+1},s_{i+2}, \ldots, s_{j-1})$ to $s_{i+1} \circ s_{i+2} \circ \ldots \circ s_{j-1} \in H$. For each $j>i+1$, there is a unique element $s=(s_{i+1},\ldots,s_{j-1}) \in G_{ij}$ with $s_l \neq 1$ for all $i<l<j$, which we call $\tau_{ij}$. We define a new $\Z$-indexed algebra $\TT$ with $\TT_{ij} = \bigoplus_{s \in G_{ij}} F_{s}$, where $F_s = F_{\overline{s}}$, with the goal of the third section being to prove the following (c.f. Theorems 3.6 and 3.12).

\begin{theorem}
    There is an isomorphism of $\Z$-indexed algebras $T(F) \xrightarrow{\sim}\TT$ between the tensor algebra $T(F)$ and $\TT$. Moreover, for all $i,j \in \Z$ there are $K_i-K_j$-bilinear isomorphisms
    \[\bbS^{nc}(F)_{ij} \cong \begin{cases}
    K_i \oplus F_{\tau_{i,i+2}} \oplus F_{\tau_{i,i+4}} \oplus \ldots \oplus F_{\tau_{ij}} & \text{ when } j-i \text{ even}, \\
    \bimod{K_i}{F}{K_{i+1}} \oplus F_{\tau_{i,i+3}} \oplus \ldots \oplus F_{\tau_{ij}} & \text{ when } j-i \text{ odd}.
    \end{cases}\]
\end{theorem}

These constructions are more general than stated, and so we expect similar decompositions to exist for other algebras constructed in a similar fashion (Van den Bergh's noncommutative $\PP^1$-bundles \cite{van_den_bergh_noncommutative_2011}, for example).

From the geometric point of view, where we see $g$ as defining a closed subscheme of $\PP^{nc}(F)$, the complement given by $\Lambda_{00}$ should be an affine curve. Algebraically, this means $\Lambda_{00}$ should be a Dedekind domain. Indeed, in the non-simple case where $\Lambda_{00}$ is a skew polynomial ring this is true, as $k[z;\psi,\delta]$ is always a Dedekind domain if $\psi$ is an automorphism \cite[\S1.2.9]{mcconnell_noncommutative_2001}. In fact, it is noted shortly after Theorem 1.1 of \cite{smertnig2017every} that most known examples of noncommutative Dedekind domains are of this form, as well as the skew Laurent series $k[z,z^{-1};\psi]$. Theorem 4.16 and Proposition 4.17 show the following.

\begin{theorem}
    If $\bimod{K_0}{F}{K_1}$ is a simple two-sided vector space, $\Lambda_{00}$ is a Dedekind domain of (Gelfand-Kirillov) dimension 1.
\end{theorem}

Finitely generated modules over noncommutative Dedekind domains have a classification theorem very similar to that of commutative Dedekind domains \cite[\S5.7.7]{mcconnell_noncommutative_2001}, which when restricted to the finite-length modules corresponds to a classification of torsion, $g$-torsion-free sheaves on $\PP^{nc}(F)$ \cite[Proposition 10.15]{chan_species_2016}.

The main result of Section 4.1 involves showing $\Lambda_{00}$ has no non-trivial idempotent ideals. In doing so we will show that two-sided ideals in $\Lambda_{00}$ correspond to free two-sided ideals of left and right rank 1 in $\bbS^{nc}(F)$. Furthermore, these free ideals of rank 1 are precisely the ideals generated by normal families of elements in $\bbS^{nc}(F)$.

\begin{theorem}
    Let $I$ be a two-sided ideal of $\bbS^{nc}(F)$ that is free of rank 1 on the left and right. Then $I$ is generated by a normal family of elements.
\end{theorem}

In fact, we show a slightly more general result, reflecting a similar fact for $\Z$-graded rings, from which the above statement follows (Corollary 4.11). The precise statement may be found in Theorem 4.7. 

A similar affine open subscheme of a noncommutative projective scheme defined by inverting a normal element arises in the study of noncommutative $\PP^1 \times \PP^1$’s, which are classified in \cite{artin_modules_1991,van_den_bergh_noncommutative_2011}. The rings corresponding to these noncommutative schemes are the \textit{cubic Artin-Schelter regular algebras}. It is shown in \cite{artin_algebras_2007} that a cubic regular algebra of dimension 3 is determined by an elliptic curve $E$ (embedded as a divisor of bi-degree $(2,2)$ in $\PP^1 \times \PP^1$), a line bundle $\cL$ on $E$ and an automorphism $\sigma \in \mathrm{Aut}(\cL)$. Similarly to the noncommutative symmetric algebra $\bbS^{nc}(F)$, a cubic regular algebra of dimension 3 contains a normal element $g$, now of degree 4, such that the ring $\Lambda_{0}$ obtained by inverting $g$ satisfies the following dichotomy \cite[Section 7]{artin_modules_1991}.

\begin{itemize}
    \item If $\sigma$ has finite order, $\Lambda_{0}$ is finite over its center.
    \item If $\sigma$ has infinite order, then $\Lambda_{0}$ is a simple ring.
\end{itemize}

A dichotomy on two-sided vector spaces $V$ also exists via the connection to representation theory afforded by the tilting sheaf constructed in \cite[Section 7]{chan_species_2016}. The endomorphism ring of this tilting sheaf is an upper triangular matrix ring $\begin{pmatrix} 
K_0 & V \\
0 & K_{1}
\end{pmatrix}$, which is precisely the `path algebra' of a $(2,2)$-bimodule in the sense of Dlab and Ringel \cite{dlab_indecomposable_1976, ringel_representations_1976}. In this setting, a $(2,2)$-bimodule (which is what we call a two-sided vector space of rank 2) $\bimod{K_0}{V}{K_1}$ is said to be \textit{algebraic} if $K_0$ and $K_1$ have a common subfield $k$ of finite index which acts centrally on $V$. For algebraic bimodules, the regular modules (the image of torsion sheaves under the derived equivalence) are well understood \cite[Theorem 1]{ringel_representations_1976}, however little is known of the non-algebraic case. In Section 5 we will see that algebraicity of the two-sided vector space $V$ determines a similar dichotomy on $\Lambda_{00}$ to that discussed above for noncommutative quadrics.

\begin{theorem}
Let $\sigma := \tau_1\tau_0$ denote the composition of the Galois involutions $\tau_i$ corresponding to the extensions $K_i\subset F$. Then
    \begin{itemize}
    \item $\bimod{K_0}{F}{K_1}$ is algebraic $\Longleftrightarrow$ $\vert \sigma \vert < \infty$ $\Longrightarrow$ $\Lambda_{00}$ is finite over its center,
    \item $\bimod{K_0}{F}{K_1}$ is non-algebraic $\Longleftrightarrow$ $\vert \sigma \vert = \infty$ $\Longrightarrow$ $\Lambda_{00}$ is a simple ring.
\end{itemize}
\end{theorem}

The proof of Theorem 1.4 is contained in Corollary 5.5 and Theorems 5.6 and 5.9. In a sense, Theorem 1.4 may be viewed as a version of \cite[Theorem 7.3]{artin_modules_1991} `over the generic point' of the noncommutative quadrics considered in \cite{artin_modules_1991}. We also remark that for non-algebraic bimodules, $\Lambda_{00}$ seems to be a new example of a simple noncommutative Dedekind domain. Indeed, known examples of simple Dedekind domains presented in both \cite{mcconnell_noncommutative_2001} and \cite{smertnig2017every} are skew polynomial rings $k[x;\psi,\delta]$ or skew Laurent rings $k[x,x^{-1};\psi]$, of which we will see $\Lambda_{00}$ appears to be neither.

\section{$\Z$-indexed algebras and noncommutative $\PP^1$'s}
In this section we will cover the necessary background on $\Z$-indexed algebras, noncommutative projective schemes and the construction of the noncommutative symmetric algebra $\bbS^{nc}(V)$ of a two-sided vector space.

\subsection{Indexed algebras}
The notion of indexed algebras first appeared in \cite{bondal1994homological,sierra_g-algebras_2011,van_den_bergh_noncommutative_2011}. Recall \cite[Section 2]{van_den_bergh_noncommutative_2011} that for a set $I$ (most commonly $I=\Z$), an \textit{$I$-indexed algebra} is a category $\A$ enriched over the category $\Ab$ of abelian groups, with $\mathrm{ob}(\A) = I$. Graded (right) $\A$-modules are then defined as additive functors $M:\A \to \Ab$. We will denote the category of graded right $A$-modules by $\Gr A$, which is simply the additive functor category $[\A,\Ab]$. 

Given an $I$-indexed algebra $\A$, the abelian groups of morphisms $A_{ij} := \Hom_{\A}(\A_j, \A_i)$ for $i,j \in I$ assemble to form a non-unital ring $A = \bigoplus_{i,j \in I} A_{ij}$. The addition on $A$ comes from the abelian group structure on each $A_{ij}$ and, for $i,j,r,l \in I$, multiplication is given by composition $A_{ij} \times A_{rl} \to A_{il}$ when $j=r$, and is $0$ if $j \neq r$. We will frequently equate $A$ with the $I$-indexed algebra $\A$, as we may recover $\A$ by taking $\mathrm{ob}(\A) = I$ and $\Hom_{\A}(j,i) = A_{ij}$ for $i,j \in I$. Similarly to the ring $A$, a graded right $\A$-module may equivalently be described as an $I$-graded abelian group $M = \bigoplus_{i \in I}M_i$ with abelian group homomorphisms $M_i \times A_{ij} \to M_j$ satisfying the right module axioms. We will work only with right modules, referring to them simply as modules, although left modules may be defined similarly in terms of contravariant functors to $\Ab$.

When $I=G$ is a group, $G$-indexed algebras generalise $G$-graded rings, as any $G$-graded ring $R=\bigoplus_{g \in G} R_g$ determines a graded Morita equivalent $G$-indexed algebra $\widehat{R} = \bigoplus_{g,h \in G} R_{g^{-1}h}$. The converse is not true, however, as there are $G$-indexed algebras not arising as $\widehat{R}$ for a $G$-graded ring $R$. In particular, $\Z$-indexed algebras are important to noncommutative projective geometry as there exist noncommutative $\PP^1 \times \PP^1$'s, in the sense of \cite{van_den_bergh_noncommutative_2011}, with $\Z$-indexed algebra coordinate rings not associated to any $\Z$-graded ring. 

While $A$ is generically non-unital, there are always `local units' $e_i \in A_{ii}$ corresponding to the identity morphism of the object $\A_i$ in $\A$. The local units $e_i$ determine right modules $e_iA$, which are precisely the rank 1 free modules over $A$ and form a set of projective generators for $\Gr A$ \cite[Section 3]{van_den_bergh_noncommutative_2011}. 

Given an $I$-indexed algebra $A$, we let $A_{\Delta} := \bigoplus_{i \in I} A_{ii}$. Then $\diag{A}$ is a sub-$I$-indexed algebra of $A$ which we call the \textit{base} of $A$. This subalgebra is analogous to the degree 0 component of a graded ring, so we say $A$ is \textit{connected} if each $A_{ii}$ is a division ring.

Given two sets $I$ and $J$, an $I$-indexed algebra $A$, and a $J$-indexed algebra $B$, we may form the (external) tensor product
\[A \otimes B := \bigoplus_{(i,j)(i',j') \in I \times J} A_{ii'} \otimes_{\Z} B_{jj'}.\]
Here external refers to the fact that these objects are $I \times J$-indexed, rather than $I$ or $J$-indexed. Using this tensor product, $A-B$-bimodules are also defined via the enriched functor category $[\A^{op} \otimes \BB,\Ab]$. An $A-B$-bimodule $T$ determines an adjoint pair of functors
$- \underline{\otimes}_A T: \Gr A \to \Gr B$ and $\underline{\Hom}_{B}(T_B,-): \Gr B \to \Gr A$ \cite{nyman_abstract_2019,mori_local_2021}, and we have that $\underline{\Hom}_A(A,M) \cong M \cong M \underline{\otimes}_A A$ for any $M \in \Gr A$.

We now insist $I=\Z$. As was pointed out in \cite{presotto_noncommutative_2016}, the abstract definition of indexed algebras allows the usual theory of Ore localisation (through formally inverting morphisms in a category) to be applied to $\Z$-indexed algebras. An element $a \in A_{ij}$ is a \textit{non-zero-divisor} if for all nonzero $b \in e_jA$ and $c \in Ae_i$, we have that $ab, ca \neq 0$. A $\Z$-indexed algebra $D$ is a \textit{domain} if every element of $D$ is a non-zero-divisor, and a graded $D$-module $M$ is \textit{uniform} if any two non-zero submodules have non-zero intersection. If a $\Z$-indexed domain $D$ is such that $e_iD$ and $De_i$ are uniform for all $i \in \Z$, then $D$ has a $\Z$-indexed ring of fractions $Q(D)$ \cite[Proposition 2.1]{chan_species_2016}. This ring of fractions is a $\Z$-indexed domain in which every non-zero morphism is invertible. A right $D$-module $M$ is (ungraded) \textit{torsion} if $M \underline{\otimes}_D Q = 0$. The kernel $t(M)$ of the localisation map $M \to M \underline{\otimes}_D Q$ is the largest (ungraded) torsion submodule of $M$, and has torsion-free quotient $M / t(M)$.

\subsection{Noncommutative projective schemes}
For the remainder of the paper we will suppose $I=\Z$, unless otherwise specified. For a $\Z$-indexed algebra $A$, we say $A$ is \textit{positively indexed} if $A_{ij} = 0$ for all $i,j \in \Z$ with $i>j$. We will suppose all our $\Z$-indexed algebras $A$ are connected and positively indexed, in which case we will denote the two-sided ideal $\bigoplus_{i<j}A_{ij}$ by $\m$. This ideal is analogous to the ideal $R_+ = \bigoplus_{n \geq 1}R_n$ of an $\N$-graded ring $R$. Furthermore, as $A$ is connected the quotients $S_i := e_iA / e_i\m\cong A_{ii}$ are simple $A$-modules. 

An $A$-module $M$ is \textit{finitely generated} if there is a surjection $\bigoplus_{l=1}^r e_{i_l}A \to M$. Similarly, $M$ is \textit{finitely presented} if $M$ is the cokernel of a morphism of free and finitely generated modules: $M \cong \mathrm{coker}(\bigoplus_{l=1}^r e_{i_l}A \to \bigoplus_{j=1}^s e_{i_j}A)$. An $A$-module $M$ is \textit{coherent} if it is finitely generated and the kernel of any map $\bigoplus_{j=1}^n e_{i_j}A \to M$ is also finitely generated. We let $\gr A$ denote the full subcategory of finitely presented $A$-modules and let $\mathsf{coh} A$ be the full subcategory of coherent $A$-modules. The latter is an abelian subcategory of $\mathsf{Gr} A$ which is closed under extensions. If all of the free modules $P_i$ and simple modules $S_i$ are coherent, we say $A$ is a (right) \textit{coherent} $\Z$-indexed algebra \cite[Section 1]{polishchuk_noncommutative_2005}, in which case $\gr A=\coh A$.

Following the ideas of Artin and Zhang \cite{artin_noncommutative_1994}, we define the (noncommutative) projectivisation of a coherent $\Z$-indexed algebra $A$ via its coherent sheaves. For such an $A$, we let $\qgr A$ be the quotient $\gr A / \mathsf{tors} A$, where $\mathsf{tors} A$ is the Serre subcategory of ($\m$-)\textit{torsion modules}: $M \in \gr A$ such that for any $m \in M$, $(mA)_n = 0$ for $n \gg 0$ \cite{polishchuk_noncommutative_2005}. We may then form the quotient $\qgr A := \gr A / \mathsf{tors} A$, which we view as the category of coherent sheaves on the noncommutative projectivisation of $A$. We denote the corresponding quotient functor by $\pi: \gr A \to \qgr A$. We also note that the assignment of a coherent module $M$ to its largest torsion submodule $\tau(M)$ is functorial, so we call $\tau:\gr A \to \mathsf{tors}A$ the \textit{torsion functor}.

\subsection{Noncommutative symmetric algebras and $\PP^1$s}
For the remainder of the paper we will suppose all our fields have characteristic $\neq 2$.

\begin{defn}{\cite[Section 3.1]{chan_species_2016}}
    Let $K_0$ and $K_1$ be fields. A two-sided $K_0-K_1$-vector space is a $K_0-K_1$ bimodule $V$.
\end{defn}

We will refer to such objects as two-sided vector spaces and bimodules interchangeably. Given a two-sided $K_0-K_1$-vector space $V$, the \textit{dimension vector} $\underline{\dim} V$ is the pair of positive integers $(\dim_{K_0} V, \dim_{K_1}V) \in \Z^+ \times \Z^+$. If $\underline{\dim} V = (n,n)$ for some $n \in \Z^+$ we say $V$ has \textit{rank }$n$. Unlike vector spaces, a two-sided vector space $V$ may be simple even when $V$ is not rank 1. If this is the case for a two-sided $K_0-K_1$-vector space $V$ of rank 2, then $V=F$ is a field and $K_0$ and $K_1$ are subfields of index 2 in $F$ \cite[Lemma 3.1(2)]{chan_species_2016}.

If $\bimod{K_0}{V}{K_1}$ has finite dimension vector $(m,n)$, then the duals
\[\ld{V} = \Hom_{K_0}(\bimod{K_0}{V}{K_1}, K_0)\]
and
\[\rd{V} = \Hom_{K_1}(\bimod{K_0}{V}{K_1}, K_1)\]
are both $K_1-K_0$ vector spaces with dimension vector $(n,m)$, however in general $\ld{V} \not\simeq \rd{V}$ as $K_1-K_0$-vector spaces. In the case that $V$ is simple and has rank 2, the left and right duals are computed in \cite{chan_species_2016}. The extensions $F/K_i$ are Galois, with Galois group $G_i = \{ 1, \tau_i \}$ and (reduced) trace map $tr_i: F \to K_i$ sending $a \in F$ to $\frac{1}{2}(a + \tau_i(a))$. These trace maps then define $K_1-K_0$-bilinear isomorphisms
\[\bimod{K_1}{F}{K_0} \xrightarrow{\sim} \ld{V}, \qquad a \mapsto tr_0(a-)\]
and
\[\bimod{K_1}{F}{K_0} \xrightarrow{\sim} \rd{V}, \qquad a \mapsto tr_1(-a).\]
To consider repeated duals we inductively define $V^{i \ast} = (V^{i-1\ast})^{\ast}$ if $i>0$. We construct a symmetric algebra of a rank 2 two-sided vector space by way of Van Den Bergh's \textit{noncommutative symmetric algebras}, introduced in \cite{van_den_bergh_non-commutative_2012}. The original construction is more general, working with rank 2 sheaves of $\cO_X-\cO_Y$ bimodules for a pair of smooth $k$-schemes $X$ and $Y$; we will work only with the case $X = \mathsf{Spec}(K_0)$ and $Y = \mathsf{Spec}(K_1)$. We begin by defining the $\Z$-indexed tensor algebra.

\begin{defn}{\cite[Section 3.2]{chan_species_2016}}
    Let $\bimod{K_0}{V}{K_1}$ be a two-sided vector space. The \textit{tensor algebra} of $V$ is the $\Z$-indexed algebra $T(V)$ where, for $i,j \in \Z$, we have 
    \[T(V)_{ij} = \begin{cases} K_{i'}, \text{ where }i' = i~ \mathrm{mod}~2, & \text{if }i=j, \\
    V^{i\ast} \otimes_{K_{i+1}} V^{i+1\ast} \otimes_{K_{i+2}} \ldots \otimes_{K_{j-1}} V^{j-1 \ast}, & \text{ if }i<j, \\
    0 & \text{ if }i>j.
    \end{cases}\]
    Multiplication is given by the tensor product $T_{ij} \otimes_{T_{jj}} T_{rl} \to T_{il}$ when $j=r$, and is $0$ otherwise.
\end{defn}

In the future, we will drop the $K$ from the tensor subscripts retaining only the integer, which again we consider mod 2 implicitly. For each $i \in \Z$, the functors $(-\otimes_{i} V^{i \ast})$ and $(- \otimes_{i+1} V^{i+1 \ast})$ form an adjoint pair, so the unit of this adjunction gives an inclusion 
\[T_{ii} \hookrightarrow T_{i,i+1} \otimes_{i+1} T_{i+1,i+2} = T_{i,i+2}.\]
We denote the image of each of these inclusions by $Q_{i} \subset T_{i,i+2}$, and we let $R$ denote the ideal of $T$ generated by the bimodules $Q_{i}$. 

\begin{defn}{\cite[Section 3.2]{chan_species_2016}}
    The \textit{noncommutative symmetric algebra} of $V$ is the quotient $\Z$-indexed algebra $\bbS^{nc}(V):=T(V) / R$.
\end{defn}

This choice of relations is unique up to graded Morita equivalence. Indeed, noncommutative symmetric algebras satisfy a form of \textit{twisting} which induces a graded Morita equivalence \cite[Section 3.2]{van_den_bergh_non-commutative_2012}, and the quotient of $T(V)$ by any non-degenerate quadratic relations (in the sense of \cite[Definition 3.1.12]{van_den_bergh_non-commutative_2012}) is related to $\bbS^{nc}(V)$ by twisting \cite[Section 4.1]{van_den_bergh_non-commutative_2012}.

Recall (\cite{van_den_bergh_noncommutative_2011}) that for $d \in \Z$ we define the \textit{shifts} $A[d]$ of a $\Z$-indexed algebra $A$ to be the $\Z$-indexed algebra defined by $A[d]_{ij} := A_{i+d,j+d}$, and we say $A$ is \textit{$d$-periodic} if there is an isomorphism $A \cong A[d]$. By again using the computation of duals from \cite{chan_species_2016}, we see that $\bbS^{nc}(V)$ is $2$-periodic if $V$ is of rank 2.

The noncommutative symmetric algebra $\bbS^{nc}(V)$ is always coherent if $\underline{\dim}(V)$ is finite \cite[Theorem 5.4(1)]{chan_representation_2019}, and is noetherian if $V$ has rank 2 \cite[Theorem 5.2]{chan_species_2016}. In either case, we may construct the noncommutative projective scheme $\PP^{nc}(V) := \qgr \bbS^{nc}(V)$. An important property of $\bbS^{nc}(V)$ for $V$ of rank 2 is the existence of a multiplicity 2 point in $\PP^{nc}(V)$, which corresponds to a \textit{normal family of elements} $g = \{g_i\}_{i \in \Z}$ of degree 2.

 \begin{defn}{\cite[Section 4]{chan_species_2016}}
    Let $A$ be a $\Z$-indexed algebra and $x = \{x_i\}_{i \in \Z}$ be a family of elements $x_i \in A_{i,i+\delta}$. We say $x$ is a \textit{normal family of elements of degree} $\delta$ if $x_iA_{i+\delta,j+\delta} = A_{ij}x_j$ for all $i,j \in \Z$.
 \end{defn}
 
We now suppose $F$ is a simple, rank 2 two-sided vector space over a pair of fields $K_0$ and $K_1$, so that the relations of $\bbS^{nc}(F)$ are generated by the images $Q_{i}$ of the unit morphisms $K_i \hookrightarrow F^{i \ast} \otimes_{i+1} F^{i+1 \ast}$. These images are rank 1 two-sided $K_i$-subspaces of $\bbS^{nc}(F)_{i,i+2}$ generated by the elements 
\[ h_{i} = 1 \otimes_{i+1} 1 + w_{i+1}^{-1} \otimes_{i+1} w_{i+1}, \]
where $w_{i+1} \in F$ is such that $w_{i+1} \not\in K_{i+1}$ but $w_{i+1}^2 \in K_{i+1}$ \cite[Section 4.1]{chan_species_2016}. Equivalently, we have that $\tau_{i+1}(w_{i+1})=-w_{i+1}$. Following \cite{chan_species_2016}, we then consider the $F$-closure $FQ_{i}$ of $Q_{i}$, which is generated by $h_{i}$ and 
\[ g_{i} := w_{i+1}h_{i} = h_{i}w_{i+1}. \] 
Then $g = \{g_i\}_{i \in \Z}$ is a normal family of elements of degree 2 in $\bbS^{nc}(F)$ \cite[Proposition 4.3]{chan_species_2016}. We will denote `powers' of $g$ as $g_i^r=g_{i}g_{i+2}\ldots g_{i+2r}$ for $i \in \Z$.

Normality tells us that the right ideal $g\bbS^{nc}(F)$ is equal to the left ideal $\bbS^{nc}(F)g$, so we view the quotient $\Z$-indexed algebra $B := \bbS^{nc}(V) / (g)$ as the coordinate ring of a closed subscheme (the \textit{point scheme}) defined by $g=0$. We mention, as is noted in \cite{chan_species_2016}, that the geometric notion of point scheme from \cite{van_den_bergh_non-commutative_2012} does not exist in this setting. However, $B$ is analogous to the `projectively commutative' quotient determined by the point scheme. The quotient $B$ is noetherian \cite[Theorem 5.2]{chan_species_2016}, so $\qgr B$ exists. Furthermore, $B_{ij} \cong F$ if $j-i \geq 1$ \cite[Lemma 5.4]{chan_species_2016}, and so $B$ has Hilbert series 
\[\sum_{n \geq 0} (\dim_{K_i}B_{i,i+n})t^n = 1 + 2t + 2t^2 + \ldots = \frac{1+t}{1-t}.\]
As such, we see that $B$ has Gelfand-Kirillov dimension 1 ($\qgr B$ is a point) and has multiplicity 2, as the dimension function $\dim_{K_i}B_{i,i+n} = 2$ for all $n>0$.

Normality also tells us that the set $\{ g_i^n \mid i \in \Z,~n \in \N \}$ satisfies the left and right Ore conditions. We may then formally invert the family $g$ to obtain the $\Z$-indexed algebra $\Lambda = \bbS^{nc}(F)[g^{-1}]$ and the `degree 0' component $\Lambda_{00}$. From the geometric perspective, $\Lambda_{00}$ is the coordinate ring of the `affine open' complement of $\qgr B$ in $\PP^{nc}(F)$. 

\begin{ex}
    We end this section by presenting a well-known example, naturally re-interpreted through the lens of noncommutative algebraic geometry: sheaves of differential operators on commutative curves. This example is a specific case of \cite[Theorem 2.4]{patrick_symmetric_2000}, which is essentially a global version of \cite[Section 5.1]{patrick_ruled_1997}. Let $C$ be an (irreducible, reduced) algebraic curve $C$ over a field $k$ of characteristic 0. For a point $p$ on $C$, we let $R = \Gamma(\cO_{C,p})$ be the local ring at $p$. Then $R$ is a discrete valuation ring, so has a unique maximal ideal $\m_p = (t)$ which is principal, where $t$ is a uniformising parameter. We will denote the fraction field of $R$ by $K$. The ring $\DD_p(C)$ of local differential operators at $p$ is an Ore extension $R[\del_t;\frac{d}{dt}]$ of $R$, and so has a natural filtration by the degree in $\del_t$: 
    \[D^i = \{ p \in \DD_p(C) \mid \deg_{\del_t}p \leq i \}.\]
    It is clear $D^0 = R$, and we have an exact sequence of $R$-bimodules
    \[ 0 \to R \to D^1 \to R\cdot \del_t \to 0. \]
    If we now pass to the generic point the above sequence becomes
    \[ 0 \to K  \to V \to K\cdot \del_t \to 0, \]
    where $V$ is a $K$-bimodule of dimension 2 on the left and right, and where the right $K$-action is given by
    \[ (p + q\del_t)f = fp + fq\del_t + \frac{df}{dt}q = f(p + q\del_t) + \frac{df}{dt}q \]
    for $p,q,f \in K$. Note that $K$ does not act centrally on $V$, nor does any finite index subfield of $K$, so $V$ is non-algebraic.
    
    The bimodule $V$ is a non-simple two-sided $K$-vector space of rank 2, as $K \subset V$ is a subbimodule. We may then construct the noncommutative symmetric algebra $\bbS^{nc}(V)$. We take a left basis $\{ z, \del_t \}$ for $V$ with $z \in K$ (which is also a right basis \cite{patrick_symmetric_2000}). Then $\bbS^{nc}(V)$ is a quotient of the $K$-bimodule tensor algebra of $V$ (thought of as noncommutative polynomials in $z$ and $\del_t$ where $K$ does not act centrally) by a quadratic relation which replaces the commutator of the generators in the usual symmetric algebra construction. Theorem 2.4 of \cite{patrick_symmetric_2000} shows that our case, any such quadratic relation in $T(V)$ has the form $(\del_tz-z\del_t-rz^2)$ for some $r \in K$. As was mentioned directly after Definition 2.3, we may suppose $r=0$ up to graded Morita equivalence (and hence also equivalence of $\PP^{nc}(V)$). Then
    \[ \bbS^{nc}(V) = K\langle z,\del_t \rangle/(zf-fz,\del_tf - f\del_t - \frac{df}{dt}z, \del_tz-z\del_t), \]
    which we recognise as the homogenisation in $z$ of the ring $K[\del_t;\frac{d}{dt}] = K \otimes_R \DD_p(C)$.

    It is shown in \cite[Section 4.2]{chan_species_2016} that the normal family of elements $g$ in $\bbS^{nc}(V)$ is determined by a single normal element $g=z \in \bbS^{nc}(V)$ (of course, this choice is obvious in our example as $z$ is central). Then $\Lambda_{00}$ is generated as a ring by $K$ and $x := \del_t z^{-1}$ \cite[Proposition 10.16]{chan_species_2016}, and for any $f \in K$ we have
    \begin{align*}
        xf = (\del_t z^{-1})f = \del_t f z^{-1} = (f\del_t + \frac{df}{dt}z)z^{-1} = fx + \frac{df}{dt}.
    \end{align*}
    It is then clear that the $\Lambda_{00}$ is a skew polynomial ring $K[x;\frac{d}{dt}]$. Compare this with the local setting of \cite[Section 5.1]{patrick_ruled_1997} where, in the case $C = \PP^1(k)$, the (homogenised) sheaf of local differential operators $\cS := H(\DD(C))$ is constructed as the noncommutative symmetric algebra of a sheaf of rank 2 $\cO_C$-bimodules. At a point $p \in C$, the stalk of $\cS$ has the form $\cS_p = R\langle z,\del_t \rangle / (\del_t t - t \del_t - z, zt-tz, z\del_t - \del_t z)$ \cite[Section 5.1]{patrick_ruled_1997}. As such, $\cS_p[z^{-1}] \cong R[\del_t z^{-1};\frac{d}{dt}] = R[x;\frac{d}{dt}]$, and hence $\Lambda_{00} = K \otimes_R \cS_p[z^{-1}]$.
\end{ex}

\section{Bimodule decompositions of tensor and noncommutative symmetric algebras}
Let $K_0,K_1$ and $F$ be fields such that $\bimod{K_0}{F}{K_1}$ is a simple, rank 2 two-sided vector space, and let $T=T(F)$ and $A=\bbS^{nc}(F)$ be the tensor algebra and noncommutative symmetric algebra as in Section 2.3. In this section we decompose the $K_i-K_j$-bimodules $T_{ij}$ and $A_{ij}$ into direct sums of irreducible subbimodules. We let $G_i = \mathrm{Gal}(F/K_i) = \{ 1, \tau_i \}$ be the corresponding Galois groups, noting that again we implicitly assume our integer subscripts are taken mod 2, and let $H=\langle \tau_0,\tau_1 \rangle$ be the subgroup of $\Aut(F)$ generated by $G_0$ and $G_1$. For $j-i > 1$, we define the set $G_{ij}:= G_{i+1} \times \ldots \times G_{j-1}$ and the map $\overline{(-)}:G_{ij} \to H$ sending $(s_{i+1},\ldots, s_{j-1}) \in G_{ij}$ to $s_{i+1} \circ s_{i+2} \circ \ldots \circ s_{j-1} \in H$. We will denote the elements $(\tau_{i+1},\tau_{i+2},\ldots,\tau_{j-1})$ by $\tau_{ij}$, which are the unique $s \in G_{ij}$ with $s_l\neq 1$ for $i<l<j$. For any $s \in G_{ij}$, we define a two-sided $K_i-K_j$-vector space $F_{s}$ with underlying abelian group $F$ and scalar multiplication $a.x.b = a\cdot x\cdot \overline{s}(b)$, where $a \in K_i$, $b \in K_j$, $x \in F$ and $\cdot$ is the usual $F$-multiplication. Letting $\TT_{ij} = \bigoplus_{s \in G_{ij}} F_{s}$ (with $\TT_{ii} = K_i$ and $\TT_{i,i+1} = \bimod{K_i}{F}{K_{i+1}}$), we will ultimately show the following (Theorems 3.6 and 3.12).

\begin{theorem}
    For all $i,j \in \Z$, there are $K_i - K_j$-bilinear isomorphisms 
    \[\mu_{ij}:T_{ij} \xrightarrow{\sim} \TT_{ij}\]
    compatible with the multiplication on $T$, giving an isomorphism of $\Z$-indexed algebras $\mu:T \xrightarrow{\sim} \TT$. Furthermore, for each $i, j \in \Z$ we have $K_i-K_j$-bilinear isomorphisms
    \[A_{ij} \cong \begin{cases}
    K_i \oplus F_{\tau_{i,i+2}} \oplus F_{\tau_{i,i+4}} \oplus \ldots \oplus F_{\tau_{ij}} & \text{ when } j-i \text{ even}, \\
    \bimod{K_i}{F}{K_j} \oplus F_{\tau_{i,i+3}} \oplus \ldots \oplus F_{\tau_{ij}} & \text{ when } j-i \text{ odd}.
    \end{cases}\]
\end{theorem}

\subsection{A $\Z$-indexed algebra isomorphic to $T$}
We begin by studying two-sided vector spaces similar to those in Theorem 3.1. Let $L$ be a field with subfields $k_{\alpha},k_{\omega}$. Given $L$-automorphisms $\rho_0$ and $\rho_1$, we define a two-sided $L$-vector space $\bimod{\rho_0}{L}{\rho_1}$ with underlying abelian group $L$ and scalar multiplication $a.x.b = \rho_0(a)\cdot x\cdot \rho_1(b)$, where $a,b,x \in L$ and $\cdot$ denotes the usual $L$-multiplication. We then restrict scalars to view $\bimod{\rho_0}{L}{\rho_1}$ as a two-sided $k_{\alpha}-k_{\omega}$-vector space.

\begin{lemma}
Let $L,k_{\alpha},k_{\omega},\rho_0,\rho_1$ be as above. Any field automorphism $f$ of $L$ gives a $k_{\alpha}-k_{\omega}$-bilinear isomorphism $\Phi_f:\bimod{\rho_0}{L}{\rho_1} \xrightarrow{\sim} \bimod{f\rho_0}{L}{f\rho_1}$.
\end{lemma}
\begin{proof}
For a field automorphism $f$ of $L$, we define $\Phi_f(x) = f(x)$ for $x \in \bimod{\rho_0}{L}{\rho_1}$. As $f$ is a field automorphism we only need to check that $\Phi_f$ preserves the $k_{\alpha}$ and $k_{\omega}$ actions. This is indeed the case, as for any $a \in k_{\alpha}$ and $b \in k_{\omega}$ we have
\[\Phi_f(a.x.b) = \Phi_f(\rho_0(a)\cdot x\cdot \rho_1(b)) = f\rho_0(a)\cdot f(x)\cdot f\rho_1(b) = f\rho_0(a)\cdot\Phi_f(x)\cdot f\rho_1(b).\]
\end{proof}

It follows that any such bimodule $\bimod{\rho_0}{L}{\rho_1}$ is isomorphic to one of the form $L_{\rho}=\bimod{1}{L}{\rho}$, where $\rho = \rho_0^{-1}\rho_1$. It is a standard result (see \cite[Chapter V, \S10.4. Corollaire]{bourbaki_algebra_2013}) that for a finite index Galois subfield $k$ of $L$ (a finite index subfield $k$ such that $L/k$ is Galois), the map 
\[x \otimes_k y \mapsto \bigoplus_{g \in \mathrm{Gal}(L/k)} x\cdot g(y)\] 
induces an isomorphism of $L$-algebras 
\[\mu:L \otimes_k L \xrightarrow{\sim} \displaystyle\bigoplus_{g \in \mathrm{Gal}(L/k)}\bimod{1}{L}{g}.\] 
Given subfields $k_{\alpha},k_{\omega} \subset L$, restricting scalars then determines a $k_{\alpha}-k_{\omega}$-bilinear isomorphism
\[ \mu^{L;k_{\alpha},k_{\omega}}_{k}:\bimod{k_{\alpha}}{L}{k} \otimes_{k} \bimod{k}{L}{k_{\omega}} \xrightarrow{\sim} \bigoplus_{g \in \mathrm{Gal}(L/k)}{}_{k_{\alpha}}(\bimod{1}{L}{g})_{k_{\omega}}. \]
We will commonly refer to this isomorphism as $\mu_k$ when $k_{\alpha},k_{\omega}$ and $L$ are fixed.

\begin{lemma}
    For $\rho_0,\rho_1 \in \Aut(L)$, $\mu_{k}^{L;k_{\alpha},k_{\omega}}$ induces a $k_{\alpha}-k_{\omega}$-bilinear isomorphism
    \[ \mu_{k}^{L;k_{\alpha},k_{\omega}}:L_{\rho_0} \otimes_{k} L_{\rho_1} \xrightarrow{\sim} \bigoplus_{g \in \mathrm{Gal}(L/k)} L_{\rho_0 g \rho_1}, \]
    defined by
    \[ x \otimes_k y \mapsto  \bigoplus_{g \in \mathrm{Gal}(L/k)} x\cdot \rho_0 g(y) \]
    for $x \in L_{\rho_0}$ and $y \in L_{\rho_1}$.
\end{lemma}
\begin{proof}
    We know $\mu_{k}^{L;k_{\alpha},k_{\omega}}$ gives an isomorphism as left $k_{\alpha}$-modules, so we only need to check that the right multiplication matches. For $x \otimes_k y \in L_{\rho_0} \otimes_k L_{\rho_1}$ and $b \in k_{\omega}$,
    \[ \mu_k((x \otimes_{k} y)b) = \mu_k(x\otimes_k (y\cdot \rho_1(b))) = \bigoplus_{g \in \mathrm{Gal}(L/k)} x(g(y \cdot\rho_1(b))). \]
    As $x \in F_{\rho_0}$, $x(g(y\cdot \rho_1(b))) = x\cdot \rho_0g(y \cdot \rho_1(b))$, and so
    \[ \mu_k((x \otimes_{k} y)b) = \bigoplus_{g \in \mathrm{Gal}(L/k)} x \cdot \rho_0 g(y) \cdot \rho_0 g\rho_1(b) = \mu_k(x \otimes_k y)b. \]
\end{proof}

\begin{lemma}
    Let $k_1,k_2 \subset L$ be finite index Galois subfields with Galois groups $G_1$ and $G_2$. For $\rho_0,\rho_1,\rho_2 \in \Aut(L)$ and any $x \in L_{\rho_0} \otimes_{k_1} L_{\rho_1} \otimes_{k_2} L_{\rho_2}$, 
    \[ \mu_{k_2}^{L;k_{\alpha},k_{\omega}} \circ (\mu_{k_1}^{L;k_{\alpha},k_2} \otimes_{k_2} 1)(x) = \mu_{k_1}^{L;k_{\alpha},k_{\omega}} \circ (1 \otimes_{k_1} \mu_{k_2}^{L;k_1,k_{\omega}})(x). \]
\end{lemma}
\begin{proof}
    Take $x = x_0 \otimes_{k_1} x_1 \otimes_{k_2} x_2 \in L_{\rho_0} \otimes_{k_1} L_{\rho_1} \otimes_{k_2} L_{\rho_2}$. Then
    \begin{align*}
        \mu_{k_2}^{L;k_{\alpha},k_{\omega}} \circ (\mu_{k_1}^{L;k_{\alpha},k_2} \otimes_{k_2} 1)(x) &= \bigoplus_{g \in G_1}\mu_{k_2}^{L;k_{\alpha},k_{\omega}}(x_0 \cdot \rho_0g(x_1) \otimes_{k_2} x_2) \\
        &= \bigoplus_{g \in G_1}\bigoplus_{h \in G_2} x_0 \cdot \rho_0g(x_1) \cdot \rho_0g\rho_1h(x_2) \\
        &= \bigoplus_{h \in G_2}\bigoplus_{g \in G_1} x_0 \cdot \rho_0g(x_1) \cdot \rho_0g\rho_1h(x_2) \\
        &= \bigoplus_{h \in G_2} \mu_{k_1}^{L;k_{\alpha},k_{\omega}}(x_0 \otimes_{k_1} x_1 \cdot\rho_1h(x_2)) \\
        &= \mu_{k_1}^{L;k_{\alpha},k_{\omega}} \circ (1 \otimes_{k_1} \mu_{k_2}^{L;k_1,k_{\omega}})(x).
    \end{align*}  
\end{proof}

In the future we will drop the $k$ from tensor subscripts, retaining only the index. We now let $\underline{k} = (k_1, \ldots, k_n)$ be an ordered tuple of finite index Galois subfields of $L$, and again let $k_{\alpha}$ and $k_{\omega}$ be arbitrary subfields of $L$. We denote the Galois groups by $G_i := \mathrm{Gal}(L/k_i)$, and let $G_{\underline{k}} := G_1 \times \ldots \times G_n$. Given an ordered tuple of $L$-automorphisms $\underline{\rho} = (\rho_0,\rho_1,\ldots,\rho_n)$, we define the $k_{\alpha}-k_{\omega}$-bimodule $L(\underline{k},\underline{\rho}) := L_{\rho_0} \otimes_{1} L_{\rho_1} \otimes_{2} \ldots \otimes_{n} L_{\rho_n}$. We also use $L(\underline{k})$ to denote $L \otimes_{1}L \otimes_2 \ldots \otimes_n L$. Repeatedly applying the isomorphisms of Lemma 3.3 gives a $k_{\alpha}-k_{\omega}$-bilinear isomorphism 
\begin{equation}
    \mu_{\underline{k}}^{k_{\alpha},k_{\omega}}:L(\underline{k},\underline{\rho}) \xrightarrow{\sim} \bigoplus_{(g_1,\ldots,g_n) \in G_{\underline{k}}} L_{\rho_0g_1\rho_1\ldots g_n\rho_n},
\end{equation} 
where $\mu_{\underline{k}}^{k_{\alpha},k_{\omega}} = \mu_{k_n}^{L;k_{\alpha},k_{\omega}} \circ (\mu_{k_{n-1}}^{L;k_{\alpha},k_n} \otimes_{n} 1) \circ \ldots \circ (\mu_{k_1}^{L;k_{\alpha},k_2} \otimes_{2} 1 \otimes_3 \ldots \otimes_{n}1)$. We will abuse notation and write this isomorphism as $\mu_{\underline{k}} = \mu_{k_n} \circ \mu_{k_{n-1}} \circ \ldots \circ \mu_{k_1}$. For any permutation $\sigma \in S_n$, where $S_n$ is the symmetric group on $n$ elements, we similarly define $\mu_{\underline{k};\sigma}^{k_{\alpha},k_{\omega}} := \mu_{k_{\sigma(n)}} \circ \mu_{k_{\sigma(n-1)}} \circ \ldots \circ \mu_{k_{\sigma(1)}}$. Thus $\mu_{\underline{k}} = \mu_{\underline{k};e}$, where $e$ is the identity in $S_n$.

\begin{propn}
    Let $\underline{k}=(k_1,\ldots,k_n)$, $k_{\alpha}$, $k_{\omega}$ and $\underline{\rho}=(\rho_0,\ldots,\rho_n)$ be as above, with $n \geq 2$. Then
    \[ \mu_{\underline{k};\sigma}^{k_{\alpha},k_{\omega}}(x) = \bigoplus_{(g_1,\ldots,g_n) \in G_{\underline{k}}} x_0 \prod_{i=1}^n \rho_0g_1\ldots \rho_{i-1}g_i(x_i) \]
    for every $x = x_0 \otimes_1 x_1 \otimes_2 \ldots \otimes_n x_n \in L(\underline{k},\underline{\rho})$ and every $\sigma \in S_n$. 
\end{propn}
\begin{proof}
    We induct on $n$, with the base case $n=2$ shown in Lemma 3.4. Suppose now that $n>2$ and that $\sigma(n) = r$. For the moment we suppose $r\neq 1,n$, as we will address these cases at the end. We let $\underline{k}' = (k_1,\ldots,k_{r-1})$ and $\underline{k}'' = (k_{r+1},\ldots,k_n)$, so that $G_{\underline{k}} = G_{\underline{k}'} \times G_r \times G_{\underline{k}''}$. For brevity we will denote $(g_1,\ldots,g_{r-1},g_{r+1},\ldots,g_n) \in G_{\underline{k}'} \times G_{\underline{k}''}$ by $(g',g'') \in G' \times G''$. As $\underline{k}'$ is an $m$-tuple for some $m<n$, and similarly for $\underline{k}''$, so the inductive hypothesis tells us
    \[\mu_{\underline{k};\sigma}^{k_{\alpha},k_{\omega}}(x) = \mu_{k_r}^{L;k_{\alpha},k_{\omega}} \circ \big(\mu_{\underline{k}'}^{k_{\alpha},k_r} \otimes_{r} \mu_{\underline{k}''}^{k_r,k_{\omega}}\big)\big((x_0 \otimes_1 \ldots \otimes_{r-1} x_{r-1}) \otimes_r (x_r \otimes_{r+1} \ldots \otimes_n x_n)\big),\]
    and that
    \begin{align*}
        &\phantom{=}\big(\mu_{\underline{k}'}^{k_{\alpha},k_r} \otimes_{r} \mu_{\underline{k}''}^{k_r,k_{\omega}}\big)\big((x_0 \otimes_1 \ldots \otimes_{r-1} x_{r-1}) \otimes_r (x_r \otimes_{r+1} \ldots \otimes_n x_n)\big) \\
        &= \biggl(\bigoplus_{g' \in G'}  x_0 \prod_{i=1}^{r-1} \rho_0g_1\ldots \rho_{i-1}g_i(x_i) \biggr) \otimes_r \biggl( \bigoplus_{g'' \in G''} x_r \prod_{j=r+1}^{n} \rho_rg_{r+1}\ldots \rho_{j-1}g_j(x_j) \biggr) \\
        &= \bigoplus_{(g',g'') \in G' \times G''}  \biggl( x_0 \prod_{i=1}^{r-1} \rho_0g_1\ldots \rho_{i-1}g_i(x_i) \biggr) \otimes_r \biggl( x_r \prod_{j=r+1}^{n} \rho_rg_{r+1}\ldots \rho_{j-1}g_j(x_j)\biggr).
    \end{align*} 
    Putting this together, we have
    \begin{align*}
        &\phantom{=} ~~\mu_{\underline{k};\sigma}^{k_{\alpha},k_{\omega}}(x) \\
        &= \mu_{k_r}^{L;k_{\alpha},k_{\omega}} \circ \big(\mu_{\underline{k}'}^{k_{\alpha},k_r} \otimes_{r} \mu_{\underline{k}''}^{k_r,k_{\omega}}\big)\big((x_0 \otimes_1 \ldots \otimes_{r-1} x_{r-1}) \otimes_r (x_r \otimes_{r+1} \ldots \otimes_n x_n)\big) \\
        &= \bigoplus_{(g',g'') \in G' \times G''} \bigoplus_{g_r \in G_r} \biggl( x_0 \prod_{i=1}^{r-1} \rho_0g_1\ldots \rho_{i-1}g_i(x_i) \biggr)\cdot \rho_0g_1\ldots \rho_{r-1}g_r \biggl( x_r \prod_{j=r+1}^{n} \rho_rg_{r+1}\ldots \rho_{j-1}g_j(x_j) \biggr) \\
        &= \bigoplus_{(g',g_r,g'') \in G_{\underline{k}}} x_0 \prod_{i=1}^{n} \rho_0g_1\ldots \rho_{i-1}g_i(x_i).
    \end{align*}
    We now consider the cases where $r=1,n$. The two cases are similar, so we only show for $r=1$. In this case $\underline{k}'$ is empty, and so $G_{\underline{k}}=G_1 \times G_{\underline{k}''}$. A near identical argument as above applies, where we now have
    \begin{align*}
        \mu_{k_1}^{L;k_{\alpha},k_{\omega}} \circ \big(1 \otimes_{1} \mu_{\underline{k}''}^{k_1,k_{\omega}}\big)\big(x_0  \otimes_1 (x_1 \otimes_{2} \ldots \otimes_n x_n)\big) &= \bigoplus_{g'' \in G''} \bigoplus_{g_1 \in G_1} x_0 \cdot \rho_0g_1\biggl( x_1 \prod_{j=2}^{n} \rho_1g_{2}\ldots \rho_{j-1}g_j(x_j) \biggr) \\
        &= \bigoplus_{(g_1,g'') \in G_{\underline{k}}} x_0 \prod_{i=1}^{n} \rho_0g_1\ldots \rho_{i-1}g_i(x_i).
    \end{align*}
\end{proof}

We now want to apply the previous results to the tensor algebra $T(\bimod{K_0}{F}{K_1})$ of a simple, rank 2 two-sided vector space $\bimod{K_0}{F}{K_1}$. For $i,j \in \Z$ with $i<j$ we let $\underline{K}_{ij} = (K_{i+1},K_{i+2},\ldots,K_{j-1})$, where again our subscripts are taken modulo 2. For $j-i>1$, (3.1) shows we have $K_i-K_j$-bilinear isomorphisms
\[\mu_{ij}:=\mu_{\underline{K}_{ij}}^{K_i,K_j}:T_{ij} = F(\underline{K}_{ij}) \xrightarrow{\sim} \displaystyle\bigoplus_{s \in G_{ij}} F_s,\]
where now $G_{ij} = \{1, \tau_{i+1}\} \times \{ 1, \tau_{i+2}\} \times \ldots \times \{1, \tau_{j-1} \}$ and $F_s = F_{\overline{s}}$. We now will set $\mu_{ij} := id_{T_{ij}}$ when $j\leq i+1$, which allows us to define a $\diag{T}$-bimodule $\TT := \mu(T) = \bigoplus_{i,j \in \Z} \mu_{ij}(T_{ij})$. Explicitly, we have 
\[\TT_{ij} = \begin{cases}
\{ 0 \} & \text{ if }j < i,\\
K_i & \text{ if }j=i, \\
\bimod{K_i}{F}{K_{i+1}} & \text{ if } j=i+1, \\
\displaystyle\bigoplus_{s \in G_{ij}}F_s & \text{ if }j>i+1.
\end{cases}\]

For $i<j<l$, $\TT_{ij} \otimes_j \TT_{jl}$ decomposes as a direct sum 
\[ \TT_{ij} \otimes_j \TT_{jl} = \bigg( \bigoplus_{s \in G_{ij}} F_s \bigg) \otimes_j \bigg( \bigoplus_{s' \in G_{jl}} F_{s'} \bigg) = \bigoplus_{(s,s') \in G_{ij} \times G_{jl}} F_s \otimes_j F_{s'}. \] 
We will again abuse notation by writing $\mu_{K_j}^{F;K_i,K_l}(\TT_{ij} \otimes_j \TT_{jl})$ to mean \break $\bigoplus_{(s,s') \in G_{ij} \times G_{jl}} \mu_{K_j}^{F;K_i,K_l}(F_s \otimes_j F_{s'})$, noting that $G_{ij}$ (resp. $G_{jl}$) is empty if $j= i+1$ (resp. $l = j+1$). 

\begin{theorem}
    Define a multiplication $\ast$ on $\TT$ where, for $a \in \TT_{ij}$ and $b \in \TT_{rl}$,
    \[a \ast b = \begin{cases} 
        ab & \text{ if }j=r \in \{i,l\}, \\
        \mu_{K_j}^{F;K_i,K_l}(a \otimes_j b), & \text{ if }j=r\not\in \{i,l\}, \\
        0, & \text{ if }j \neq r.       
    \end{cases}\]
    Then $\TT$ is a positively $\Z$-indexed algebra, and $T \cong \TT$ as $\Z$-indexed algebras.
\end{theorem}
\begin{proof}
    By construction we have that $\TT_{ij}$ is isomorphic to $T_{ij}$ for all $i,j \in \Z$. As $T$ is a positively $\Z$-indexed algebra, the same is true of $\TT$ if
    \[ \mu_{ij}(x) \ast \mu_{jl}(y) = \mu_{il}(x \otimes_j y) \]
    for all $x \in T_{ij}$ and $y \in T_{jl}$. This is clear when $j\neq r$, and when $j=r\in \{i,l\}$. When $j=r\not\in \{i,l\}$ this follows from Proposition 3.5, as
    \[ \mu_{ij}(x) \ast \mu_{jl}(y) = \mu_{K_j}^{F;K_i,K_l} \circ \big(\mu_{ij} \otimes_j \mu_{jl}\big)\big(x \otimes_j y\big) = \mu_{il}(x \otimes_j y). \]
\end{proof}

We note that as $G_j = \{1,\tau_j\}$, the multiplication on $\TT$ may be described as
\[a_s \ast b_{s'} = a_s \overline{s}(b_{s'}) \oplus a_s \overline{s}\tau_j(b_{s'}) \in F_{(s,1,s')} \oplus F_{(s,\tau_j,s')} \subset \TT_{il}\] 
for $i<j<l$, any $s \in G_{ij}$, $s' \in G_{jl}$, and any $a_s \in F_s \subset \TT_{ij}$, $b_{s'} \in F_{s'} \subset \TT_{jl}$. We will frequently use this description of the multiplication on $\TT$ throughout.

As a reminder from Section 2.3, $Q_i$ is the two-sided $K_{i}$-subspace of $T_{i,i+2}$ given by the image of the unit morphism $K_{i} \hookrightarrow T_{i,i+2}$, and is generated by the element $h_{i}$. Applying $\mu_{i,i+2}$ to $h_{i}$ gives the following.

\begin{cor}
Let $\cR$ be the ideal of $\TT$ generated in degree 2 by the elements $h_i':=(2, 0) \in F_1 \oplus F_{\tau_{i+1}} = \TT_{i,i+2}$. Then the noncommutative symmetric algebra $A$ is isomorphic to $\TT/\cR$.
\end{cor}
\begin{proof}
    Recall once again from Section 2.3 that the relations of $A$ are generated by elements $h_i=1 \otimes_{i+1} 1 + w_{i+1}^{-1} \otimes_{i+1} w_{i+1} \in T_{i,i+2}$, where $w_{i+1} \in F$ is such that $\tau_{i+1}(w_{i+1}) = -w_{i+1}$. Applying $\mu$ then gives
    \[ \mu_{i,i+2}(h_i) = (1 + w_{i+1}^{-1}w_{i+1}, 1 + w_{i+1}^{-1}\tau_{i+1}(w_{i+1})) = (2, 1 - w_{i+1}^{-1}w_{i+1}) = (2,0). \]
\end{proof}

We also mention that the elements of the normal family $g = \{g_i \in A_{i,i+2}\}_{i \in \Z}$ have images $g_i' := \mu_{i,i+2}(g_i) = (2w_{i+1},0) \in F \oplus F_{\tau_{i+1}}$.

\subsection{A bimodule decomposition of $\bbS^{nc}(F)$}
Theorem 3.6 gives a decomposition of the components $T_{ij}$ of the tensor algebra into a direct sum of irreducible subbimodules. We now aim to give a similar description of the quotient $\A := \TT/\cR$ defined in Corollary 3.7, which is the image of the noncommutative symmetric algebra $A$ under the isomorphism $\mu$.

For $l \in \Z$, we let $\TT_1^{(l)}$ denote the $K_{l}-K_{l+2}$-subbimodule $F_1$ of $\TT_{l,l+2}$. As $\TT_1^{(l)}$ has a simultaneous left and right $K_{l}$-basis $\{ h_l',g_l'\}$, we see $\TT_1^{(l)} \cong (K_{l})^2$. For $i<l<j-2$ and $s=(s_{i+1},\ldots,s_{j-1}) \in G_{ij}$ with $s_{l+1} = 1$, we set $s' = (s_{i+1},\ldots,s_{l})$ and $s'' = (s_{l+2},\ldots,s_{j-1})$. We then have
\begin{equation}
    \resizebox{.9\hsize}{!}{$F_{s'} \ast \TT_1^{(l)} \ast F_{s''} \subseteq F_{(s',1,1,1,s'')} \oplus F_{(s',\tau_{l},1,1,s'')} \oplus F_{(s',1,1,\tau_{l+2},s'')} \oplus F_{(s',\tau_{l},1,\tau_{l+2},s'')}$}.
\end{equation}
When $l=i$ or $l=j-2$, we instead have
\begin{equation}
    \TT_1^{(i)} \ast F_{s''} \subseteq F_{(1,1,s'')} \oplus F_{(1,\tau_{i+2},s'')}
\end{equation}
and
\begin{equation}
    F_{s'} \ast \TT_1^{(j-2)} \subseteq F_{(s',1,1)} \oplus F_{(s',\tau_{j-2},1)},
\end{equation}
respectively. 

\begin{lemma}
    The inclusions (3.2), (3.3) and (3.4) are equalities.
\end{lemma}
\begin{proof}
    It suffices to show that both sides of each inclusion have the same left and right dimensions. We first suppose $i<l<j-2$. By the definition of the multiplication on $\TT$ we have an isomorphism of $K_i-K_j$-bimodules 
    \[ \mu_{ij} \circ (\mu_{il}^{-1} \otimes_l \mu_{l,l+2}^{-1} \otimes_{l+2} \mu_{l+2,j}^{-1}): F_{s'} \otimes_{l} \TT_1^{(l)} \otimes_{l+2} F_{s''} \xrightarrow{\sim} F_{s'} \ast \TT_1^{(l)} \ast F_{s''}, \] and we know $\TT_1^{(l)}$ is isomorphic to $(K_{l})^2$, so 
    \[F_{s'} \ast \TT_1^{(l)} \ast F_{s''} \cong \left( F_{s'} \otimes_{l} F_{s''}\right)^{\oplus 2}.\] 
    Now $F_{s'} \otimes_{l} F_{s''}$ has left and right dimension 4, so it follows that the LHS of (3.2) has left and right dimension 8. The RHS of (3.2) also has left and right dimension 8, so (3.2) is an equality. Similar arguments show that both sides of (3.3) and (3.4) have left and right dimension 4, and so (3.3) and (3.4) must also be equalities.
\end{proof}

We now let $f_{\cR,i}(n) = \dim_{K_i} \cR_{i,i+n}$ denote the Hilbert function of $e_i\cR$. We will also use the same notation $f_{\TT,i}(n)$ and $f_{\A,i}(n)$ for the Hilbert functions of $e_i\TT$ and $e_i\A$, respectively. Clearly $f_{\TT,i}(n)=2^n$, and it is known that $f_{\A,i}(n)=n+1$ \cite[Lemma 3.6(2)]{chan_species_2016}. We give an independent proof that $f_{\A,i}(n)=n+1$ by showing $f_{\cR,i}(n) = 2^n-n-1$ (the $n$-th Eulerian number $E(n,1)$ \cite{OEIS}), as $f_{\A,i}(n) = f_{\TT,i}(n) - f_{\cR,i}(n)$. We do so by examining the intersection of the images of the two-sided subspaces $\QQ_i := \mu_{i,i+2}(Q_i)$ inside the relations $\cR_{ij}$. For $j-i \geq 2$ and $i \leq l \leq j-2$, we let $\cR_{ij}^{(l)} = \TT_{il} \ast \QQ_l \ast \TT_{l+2,j} \subset \cR_{ij}$. As the $K_l-K_{l+2}$-bimodule $\QQ_l$ has a complement $ \QQ_l^{\perp} := K_l (g_l') K_{l+2} \oplus F_{\tau_{l+1}}$ in $\TT_{l,l+2}$, $\cR_{ij}^{(l)}$ has a complement $\cR_{ij}^{(l)\perp} := \TT_{il} \ast \QQ_l^{\perp} \ast \TT_{l+2,j}$ in $\TT_{ij}$.

\begin{lemma}
For $j-i \geq 3$ and $i \leq l < l' \leq j$,
\[\cR_{ij}^{(l)} \cap \cR_{ij}^{(l')} = \begin{cases}
    \{ 0 \} & \text{ if } l'-l = 1, \\
    \TT_{il} \ast \QQ_l \ast \TT_{l+2,l'} \ast \QQ_{l'} \ast \TT_{l'+2,j} & \text{ otherwise}.
\end{cases}\]
\end{lemma}
\begin{proof}
We first suppose $l'-l>1$, in which case we write $\TT_{ij}$ as $\TT_{i,l+2} \ast \TT_{l+2,l'} \ast \TT_{l'j}$. We then have $\cR_{ij}^{(l)} = \cR_{i,l+2}^{(l)} \ast \TT_{l+2,l'} \ast \TT_{l'j}$ and $\cR_{ij}^{(l')} = \TT_{i,l+2} \ast \TT_{l+2,l'} \ast \cR_{l'j}^{(l')}$. Then
\[\TT_{ij} = \TT_{i,l+2} \ast \TT_{l+2,l'} \ast \TT_{l'j} = \big(\cR_{i,l+2}^{(l)}\oplus \cR_{i,l+2}^{(l)\perp}\big) \ast \TT_{l+2,l'} \ast \big(\cR_{l'j}^{(l')} \oplus \cR_{l'j}^{(l')\perp}\big),\]
and hence
\[\cR_{ij}^{(l)} = \Big( \cR_{i,l+2}^{(l)} \ast \TT_{l+2,l'} \ast \cR_{l'j}^{(l')}\Big) \oplus \Big( \cR_{i,l+2}^{(l)} \ast \TT_{l+2,l'} \ast \cR_{l'j}^{(l') \perp}\Big),\]
\[\cR_{ij}^{(l')} = \Big( \cR_{i,l+2}^{(l)} \ast \TT_{l+2,l'} \ast \cR_{l'j}^{(l')}\Big) \oplus \Big( \cR_{i,l+2}^{(l)\perp} \ast \TT_{l+2,l'} \ast \cR_{l'j}^{(l')}\Big).\]
It is then clear that
\[\cR_{ij}^{(l)} \cap \cR_{ij}^{(l')} = \cR_{i,l+2}^{(l)} \ast \TT_{l+2,l'} \ast \cR_{l'j}^{(l')} = \TT_{il} \ast \QQ_l \ast \TT_{l+2,l'} \ast \QQ_{l'} \ast \TT_{l'+2,j}.\]
Now if $l'=l+1$ we may suppose $j-i=3$ without loss of generality, as 
\[\cR_{ij}^{(l)} \cap \cR_{ij}^{(l+1)} = \TT_{il} \ast (\QQ_l \ast \TT_{l+2,l+3} \cap \TT_{l,l+1} \ast \QQ_{l+1}) \ast \TT_{l+3,j}.\] 
Then $\QQ_i \ast \TT_{i+2,i+3} \subset F_{(1,1)} \oplus F_{(1, \tau_{i+2})}$ and $\TT_{i,i+1} \ast \QQ_{i+1} \subset F_{(1,1)} \oplus F_{(\tau_{i+1}, 1)}$, so we see 
\[\QQ_{i} \ast \TT_{i+2,i+3} \cap \TT_{i,i+1} \ast \QQ_{i+1} \subset F_{(1,1)}.\] 
Suppose such an element were to exist, so that $h_{i}' \ast a = b \ast h_{i+1}'$ for some $a \in \TT_{i+2,i+3}$ and $b \in \TT_{i,i+1}$. As $h_{i}' \ast a = 2a \oplus 2\tau_{i+2}(a) \in F_{(1,1)} \oplus F_{(1,\tau_{i+2})}$, we must have $\tau_{i+2}(a)=0$. It follows that $a=0$, and a similar argument shows $b = 0$, so $\QQ_{i} \ast \TT_{i+2,i+3} \cap \TT_{i,i+1} \ast \QQ_{i+1} = \{ 0 \}$. 
\end{proof}

\begin{propn}
    For $n \geq 0$, $f_{\cR,i}(n) = 2^n-n-1$.
\end{propn}
\begin{proof}
    This is trivial for $n< 2$. We proceed by induction on $n$, supposing now that $n \geq 2$ and $f_{\cR,i}(m) = 2^{m}-m-1$ for all $m<n$. We also let $j=i+n$. In the notation of Lemma 3.9 we have 
    \[\cR_{ij} = \sum_{l=i}^{j-2} \cR_{ij}^{(l)} = (\cR_{i,j-1} \ast \TT_{j-1,j}) + \cR_{ij}^{(j-2)}, \] 
    and
    \[\dim_{K_i} \cR_{ij}^{(l)} = \dim_{K_i} \TT_{il} \ast \QQ_{l} \ast \TT_{l+2,j} = ( \dim_{K_i} \TT_{il} ) ( \dim_{K_{l+2}} \TT_{l+2,j} ) = 2^{n-2}.\]
    It follows from Lemma 3.9 that $\cR_{i,j-1} \ast \TT_{j-1,j} \cap \cR_{ij}^{(j-2)} = \cR_{i,j-2} \ast \QQ_{j-2}$ (this is also \cite[Lemma 3.6(1)]{chan_species_2016}), so
    \begin{align*}
        f_{\cR,i}(n) &= \dim_{K_i} (\cR_{i,j-1} \ast \TT_{j-1,j}) + \dim_{K_i} (\cR_{ij}^{(j-2)}) - \dim_{K_i} (\cR_{i,j-2} \ast \QQ_{j-2}) \\
        &= 2f_{\cR,i}(n-1) + 2^{n-2} - f_{\cR,i}(n-2).
    \end{align*}
    We then apply the inductive hypothesis to see
    \[ f_{\cR,i}(n) = 2(2^{n-1}-n) + 2^{n-2} - (2^{n-2}-n+1) = 2^n - n-1. \]
\end{proof}

\begin{remark}
    Using Lemma 3.9 we may compute $f_{\cR,i}(n)$ directly, which provides a new proof of a known formula for the Eulerian numbers $E(n,1)$. A non-empty $m$-fold intersection $\cR_{i,i+n}^{(l_1)} \cap \ldots \cap \cR_{i,i+n}^{(l_m)}$ has dimension $2^{n-2m}$. For such an intersection we define a binary sequence $(c_q)_{q=1}^{n-1}$ of length $n-1$, where 
    \[c_q = \begin{cases}
    0, & q \not\in \{ l_1, \ldots, l_m\}, \\
    1, & q \in \{ l_1, \ldots, l_m\}.
    \end{cases}\] 
    Lemma 3.9 then tells us that if the intersection is non-empty, there is no $q$ for which $c_q=c_{q+1} = 1$. So counting non-empty $m$-fold intersections is the same as counting binary strings of length $n-1$ with $m$ nonconsecutive 1's, of which there are $\binom{n-m}{m}$ such strings \cite{muir_combinations_1904}. We then use inclusion-exclusion to compute 
    \begin{align}
        E(n,1) = f_{\cR,i}(n) =  \sum_{m=1}^{\lfloor n/2 \rfloor} (-1)^{m-1} \binom{n-m}{m}2^{n-2m}.
    \end{align}
    After rearranging and re-indexing, we recover the known formula (see \cite{OEIS})
    \[ E(n,1) = \sum_{m=0}^{\lfloor (n-1)/2 \rfloor} (-1/2)^{m} \binom{n-m+1}{m-1}2^{n-m-2}. \]
\end{remark}

We now let $\mathcal{I}$ denote the ideal of $\TT$ generated by $\TT_1^{(l)}$ for $l \in \Z$. As $\TT_1^{(l)}$ is the span of the elements $h_l'$ and $g_l'$, the quotient $\TT/\mathcal{I}$ is the image of $B=A/gA$ under the isomorphism $\mu: T \xrightarrow{\sim} \TT$. Lemma 3.8 shows that, for $i\leq l \leq j-2$, we have
\[\TT_{il}\ast \TT_1^{(l)} \ast \TT_{l+2,j} = \displaystyle\bigoplus_{\substack{s \in G_{ij} \\ s_{l+1}=1}}F_{s},\]
so we see $\mathcal{I}_{ij} = \bigoplus_{\substack{s \in G_{ij}\setminus \{\tau_{ij}\}}} F_s$, recalling that $\tau_{ij}=(\tau_{i+1},\tau_{i+2},\ldots,\tau_{j-1}) \in G_{ij}$ is the unique element $s \in G_{ij}$ with $s_{l+1}\neq 1$ for all $i\leq l\leq j-2$.

\begin{theorem}
    For all $i\leq j$, there are $K_i-K_j$-bilinear isomorphisms 
    \[\A_{ij} \cong \begin{cases}
    K_i \oplus F_{\tau_{i,i+2}} \oplus F_{\tau_{i,i+4}} \oplus \ldots \oplus F_{\tau_{ij}} & \text{ when } j-i \text{ even} \\
    \bimod{K_i}{F}{K_{i+1}} \oplus F_{\tau_{i,i+3}} \oplus \ldots \oplus F_{\tau_{ij}} & \text{ when } j-i \text{ odd}.
    \end{cases}\]
\end{theorem}
\begin{proof}
    We induct on $n=j-i$, with the cases $n=0,1$ being trivial. Suppose $n \geq 2$. As $\mathcal{I}_{ij} = \bigoplus_{\substack{s \in G_{ij}\setminus \{\tau_{ij}\}}} F_s$, the map $\TT_{ij} \to (\TT/\mathcal{I})_{ij}$ is the projection onto the summand $F_{\tau_{ij}}$. Thus $\TT_{ij} \cong \mathcal{I}_{ij} \oplus F_{\tau_{ij}}$, and as $\cR\subset \mathcal{I}$ we see $\A_{ij} \cong \mathcal{I}_{ij}/\cR_{ij} \oplus F_{\tau_{ij}}$. We claim $\mathcal{I}_{ij}/\cR_{ij} \cong \A_{i,j-2}$, from which the theorem will follow by the inductive hypothesis. Right multiplication by the element $g_{j-2}' \in \TT_1^{(j-2)}$ gives an inclusion $\iota_g:\A_{i,j-2}\hookrightarrow\A_{ij}$. Note that this is actually right $K_j$-linear, as $K_{j-2}=K_j$ and $ag_{j-2}'=g_{j-2}'a$ for all $a \in K_j$. The subbimodule $F_{\tau_{ij}}$ is contained in $\coker \iota_g$ as $F_s \ast g_{j-2}' \subset F_{(s,1,1)} \oplus F_{(s,\tau_{j-2},1)}$ for every $s \in G_{i,j-2}$, and so $\A_{i,j-2}g_{j-2}' \subset \mathcal{I}_{ij}/ \cR_{ij}$. By Proposition 3.10 we have 
    \[\dim_{K_i} (\A_{ij}) - \dim_{K_i} (\A_{i,j-2}g_{j-2}') = n+1 - (n-1) = 2 =\dim_{K_i} (F_{\tau_{ij}}).\]
    Thus $\coker \iota_g \cong F_{\tau_{ij}}$, and hence $\mathcal{I}_{ij}/\cR_{ij} \cong \A_{i,j-2}$. 
\end{proof}

\begin{remark}
    The construction of the $\Z$-indexed algebra $\TT$ is very general, hence we expect a similar decomposition should exist for other algebras constructed by way of noncommutative tensor/symmetric algebras. As such, this decomposition may be a useful tool with wider applications.
\end{remark}


\section{{$\Lambda_{00}$} is a noncommutative affine curve}
Normality of the family $g$ tells us that the set $\{ g_i^n = g_i g_{i+2} \ldots g_{i+2n} \mid i \in \Z, n \geq 0 \}$ satisfies the left and right Ore conditions, so we may invert the members of the family $g$ to obtain a two-sided $\Z$-indexed algebra of fractions $\Lambda$. The $\Z$-indexed algebra $\Lambda$ is \textit{strongly graded}, so that $\gr \Lambda$ is equivalent to the category $\mathsf{mod} \Lambda_{00}$ of finitely generated $\Lambda_{00}$-modules \cite[Lemma 10.13]{chan_species_2016}, and hence $\qgr \Lambda = \gr \Lambda$. As $B = A / gA$ is defined as the quotient of $A$ by a `homogeneous' ideal, we view the noncommutative Proj $\qgr B$ as a closed subscheme of $\PP^{nc}(F)$, with $\Lambda_{00}$ being the coordinate ring of the corresponding affine open complement. In this section we show that $\Lambda_{00}$ should be viewed as a noncommutative affine curve. In the commutative setting, this corresponds algebraically to the coordinate ring being a \textit{Dedekind domain}. As in \cite[\S 5.2.6]{mcconnell_noncommutative_2001}, we say a prime Goldie ring $R$ is an \textit{Asano order} if every non-zero two-sided ideal $I \subset R$ is invertible in the ring of fractions $Q(R)$.

\begin{defn}{\cite[\S5.2.10]{mcconnell_noncommutative_2001}}
    A (possibly noncommutative) ring $R$ is a \textit{Dedekind domain} if it is a hereditary noetherian domain and an Asano order.
\end{defn}

When $R$ is commutative, this definition coincides with the usual one. As with commutative Dedekind domains, noncommutative Dedekind domains have a number of equivalent characterisations. In particular, the following will be of use {\cite[Proposition 5.6.3]{mcconnell_noncommutative_2001}}: A hereditary noetherian domain $R$ is a Dedekind domain if and only if the only idempotent (two-sided) ideals in $R$ are $0$ and $R$ itself.

While any commutative Dedekind domain must have (Gelfand-Kirillov) dimension 1, the same is not true of noncommutative Dedekind domains. For example, the Weyl algebra $A_1(k) = k[x,\frac{d}{dx}]$ is a noncommutative Dedekind domain of dimension 2. Section 4 is devoted to showing the following (cf. Theorem 4.16 and Proposition 4.17).

\begin{theorem}
    $\Lambda_{00}$ is a Dedekind domain of dimension 1.
\end{theorem}

As such, we are justified in calling $\Lambda_{00}$ a noncommutative affine curve. This result immediately allows us to characterise finitely-generated $\Lambda_{00}$-modules using the general ideal theory for Dedekind domains \cite[\S7]{mcconnell_noncommutative_2001}. In particular, this classifies finite-length $\Lambda_{00}$-modules, which correspond to torsion, $g$-torsion-free sheaves on $\PP^{nc}(F)$ \cite[Proposition 10.15]{chan_species_2016}.

\begin{remark}
    Theorem 4.2 is already known for non-simple two-sided vector spaces $V$. We have a concrete description of $\Lambda_{00}$ as a skew polynomial ring $k[z;\psi,\delta]$ \cite[Proposition 10.16]{chan_species_2016}, and $k[z;\psi,\delta]$ is a Dedekind domain when $\psi$ is an automorphism \cite[\S1.2.9]{mcconnell_noncommutative_2001}. 
\end{remark}

\subsection{Two-sided ideals in $\Z$-indexed algebras}
We begin by studying two-sided ideals in $\Z$-indexed algebras. Recall from Definition 2.4 that a collection of elements $x=\{x_i \in C_{i,i+\delta}\}_{i \in \Z}$ in a $\Z$-indexed algebra $C$ is a normal family of degree $\delta$ if $x_iC_{i+\delta,j+\delta} = C_{ij}x_{j}$ for all $i,j \in \Z$. Given a normal family $x$, the two-sided ideal generated by the elements $x_i$ will be denoted by $(x)$. Suppose that $C$ is connected and positively indexed, and call the division rings $D_i := C_{ii}$. We will also suppose that $\dim_{D_i}C_{ij} > 1$ and $\dim_{D_j}C_{ij} > 1$ if $i < j$. A right ideal $I \subset C$ is \textit{free of rank 1} if there is a family of regular elements $\{a_i \in C_{i,i+\delta_i}\}_{i \in \Z}$ such that $e_iI = a_iC$ for all $i \in \Z$. If $I$ is two-sided and free of rank 1 on both sides, then there are two families $\{a_i \in C_{i,i+\delta_i}\}_{i \in \Z}$ and $\{b_j \in C_{j-\varepsilon_j,j}\}_{j \in \Z}$ with $e_iI = a_iC$ and $Ie_j = Cb_j$ for all $i,j \in \Z$. These families have associated `translation functions' $t_r,t_l: \Z \to \Z$ defined by $t_r:i \mapsto i + \delta_i$ and $t_l:j \mapsto j - \varepsilon_j$.

\begin{lemma}
    The functions $t_r$ and $t_l$ are inverse, and hence are bijective.
\end{lemma}
\begin{proof}
    As $a_i \in I_{i, t_r(i)}=a_iD_{t_r(i)}$ and $I$ is a two-sided ideal, left multiplication by $D_i$ gives an inclusion $D_ia_i \subset a_iD_{t_r(i)}$. Similarly, $b_{j}D_j \subset D_{t_l(j)}b_j$. As $I_{i, t_r(i)} = C_{i, t_lt_r(i)}b_{t_r(i)}$, there is some $x \in C_{i, t_lt_r(i)}$ such that $a_i = xb_{t_r(i)}$. If $t_lt_r(i) \neq i$, we must have 
    \[C_{i, t_lt_r(i)}b_{t_r(i)} = I_{i,t_r(i)} = a_iD_{t_r(i)} = xb_{t_r(i)}D_{t_r(i)} \subset xD_{t_lt_r(i)}b_{t_r(i)}.\]
    Now each $b_{t_r(i)}$ is regular, so we conclude $C_{i, t_lt_r(i)} \subset xD_{t_lt_r(i)}$. Our assumptions on dimension ensure $\dim_{D_i} C_{i,t_lt_r(i)} = 1$ only if $t_lt_r(i) =i$, and the right dimension of $xD_{t_lt_r(i)}$ is 1, so we must have $t_lt_r(i) =i$. A similar argument shows that $t_rt_l(j)=j$ for all $j \in \Z$, so $t_r$ and $t_l$ are inverse functions. 
\end{proof}

It follows that $\delta_i = \varepsilon_{t_r(i)}$ for all $i \in \Z$, and so we have that $a_i=r_ib_{t_r(i)}$ for some $r_i \in D_{i}$.

\begin{propn}
    For all $i,j \in \Z$, $\delta_i = \delta_j$.
\end{propn}
\begin{proof}
    We claim $t_r$ is strictly increasing. To see this, we suppose there is an $i$ for which $t_r(i+1) \leq t_r(i)$. Then 
    \[\{ 0\} \neq a_{i+1}C_{t_r(i+1), t_r(i)} = I_{i+1, t_r(i)} \subset Ie_{t_r(i)}.\]
    But $Ie_{t_r(i)} = Cb_{t_r(i)}$ where $b_{t_r(i)} \in C_{i, t_r(i)}$, and hence $I_{i+1, t_r(i)} = C_{i+1, i}b_{t_r(i)} = \{0\}$, as $C$ is positively indexed. This is a contradiction, and as such we must have that $t_r$ is strictly increasing. The following lemma then completes the proof of the proposition.
\end{proof}

\begin{lemma}
    A function $f: \Z \to \Z$ is bijective and strictly increasing if and only if $f(i)=i+n$ for some $n \in \Z$ and all $i \in \Z$.
\end{lemma}
\begin{proof}
    Clearly $i\mapsto i+n$ is a strictly increasing bijection for $n \in \Z$. For the other direction, we first note that $f(i) = i + n_i$ for each $i \in \Z$, where $n_i := f(i)-i$. Now take any $j \in \Z$ with $j>i$. As $f$ is a bijection, for each $m$ in the interval $[f(i),f(j)]$ there is some $m' \in \Z$ with $f(m') = m$, and as $f$ is strictly increasing we must have $i \leq m' \leq j$. It then follows that $[f(i),f(j)] = \{f(i),f(i+1),\ldots,f(j)\}$, so that
    \[ f(j)-f(i) = \vert [f(i),f(j)]\vert-1 = \vert[i,j]\vert-1 = j-i. \]
    After rearranging we see that $n_j=n_i=:n$, and so $f(i) = i+n$ for all $i \in \Z$.
\end{proof}

In light of Proposition 4.5 we will refer to $\delta_i$ by $\delta$, so that $t_r(i) = i + \delta$ and $t_l(j) = j-\delta$.

\begin{theorem}
    If a two-sided ideal $I$ of $C$ is free of rank 1 on both sides, then $I$ is generated by a normal family of elements.
\end{theorem}
\begin{proof}
    If $I$ is generated by families $a=\{a_i\}$ and $b=\{b_j\}$ on the right and left, respectively, then Lemma 4.4 and Proposition 4.5 show that there is some $\delta \in \Z$ such that $a_i \in C_{i,i+\delta}$ and $b_j \in C_{j-\delta, j}$ for all $i,j \in \Z$. Then there is some $0\neq r \in D_{j}$ such that $a_j = rb_{j+\delta}$, and $C_{ij}r = C_{ij}$ for all $i$ as $D_j$ is a division ring, so 
    \[I_{i,j+\delta} = C_{ij}b_{j+\delta} = (C_{ij}r)b_{j+\delta} = C_{ij}a_j\]
    for all $i \in \Z$. Thus the $a_j$ are also left generators for $I$, and hence $a$ is a normal family as
    \[ a_iC_{i+\delta,j+\delta} = I_{i,j+\delta} = C_{ij}a_{j}. \]
\end{proof}

\subsection{$\Lambda_{00}$ is a Dedekind domain}
We now begin working towards showing that the ring $\Lambda_{00} = A[g^{-1}]_{00}$ is a Dedekind domain. Throughout this section $I$ will be a two-sided ideal of $\Lambda_{00}$. Such an ideal determines ideals $\widetilde{I} = \Lambda I \Lambda$ of $\Lambda$ and $\overline{I} = \widetilde{I} \cap A$ of $A$. We will write elements of $\Lambda$ as $ag^{-r}$, where $a \in A_{ij}$ for some $i,j \in \Z$ and $ag^{-r}=a(g_{j-2r}g_{j-2(r-1)}\ldots g_{j-2})^{-1} \in \Lambda$. Recall that an ideal $J$ of $A$ is right (resp. left) $g$\textit{-saturated} if for $x \in A$, $xg \in J$ (resp. $gx \in J$) implies $x \in J$.

\begin{lemma}
$\overline{I}[g^{-1}] = \widetilde{I}$, and $\overline{I}$ is left and right $g$-saturated.
\end{lemma}
\begin{proof}
For the first statement we suppose $a \in \overline{I}$ and $ag^{-r} \in \overline{I}[g^{-1}]$. By definition we have $a \in \widetilde{I}$, so $ag^{-r} \in \widetilde{I}$ as $g^{-r} \in \Lambda$. Conversely, if $a \in A$ and $ag^{-r} \in \widetilde{I}$ we have $(ag^{-r})g^r = a \in \widetilde{I}$, and so $a \in \overline{I}$. Hence $ag^{-r} \in \overline{I}[g^{-1}]$.

Now if $x \in A$ and $xg \in \overline{I}$, then $(xg)g^{-1} = x \in \widetilde{I}$, and so $x \in \overline{I}$. Thus $\overline{I}$ is right $g$-saturated. A similar argument shows left $g$-saturation.
\end{proof}

Recall from Section 2.2 that a positively $\Z$-indexed algebra $C$ contains a maximal ideal $\mathfrak{m} = C_{\geq 1}$, and that a $C$-module $M$ is $\mathfrak{m}$-torsion if every element of $M$ is annihilated by some power of $\m$. We also recall that the largest $\m$-torsion submodule of $M$ is denoted by $\tau(M)$.

\begin{lemma}
$A / \overline{I}$ is torsion-free on both sides.
\end{lemma}
\begin{proof}
Suppose $x + \overline{I} \in A / \overline{I}$ is torsion. Then, in particular, there is some $n>0$ such that $x g^n \in \overline{I}$. But $\overline{I}$ is $g$-saturated by the previous lemma, so $x \in \overline{I}$ and hence $x+ \overline{I} = 0 \in A/\overline{I}$.
\end{proof}

Every noetherian $A$-module $M$ has a finite resolution by sums of the rank 1 free modules $e_iA$, and so has Hilbert function $f_M(n) = \dim_{K_n}M_n$ of the form $cn+d$ for $n\gg 0,~c \in \Z^+,~d \in \Z$. The leading coefficient $c$ is the \textit{multiplicity} of $M$, and $\dim M = \deg f_M(n) + 1$ if $f_M(n)$ is not eventually 0, where $\dim$ is the Gelfand-Kirillov dimension. As in \cite[Section 6]{chan_species_2016}, we say that a noetherian $A$-module $M$ is \textit{critical} if $\dim(M/N) < \dim(M)$ for all proper submodules $N \subset M$. If $M$ is critical and $\dim M \geq 1$, any submodule $N \subset M$ must have the same multiplicity as $M$. Indeed, $f_{M/N}(n) = f_M(n) - f_N(n)$, so $\deg f_{M/N}(n) < \deg f_M(n)$ only if $f_M(n)$ and $f_N(n)$ have the same degree and leading coefficient.

\begin{propn}
Each $e_i\overline{I}$ is free of rank 1.
\end{propn}
\begin{proof}
Consider the $A$-modules $M_i = e_i A / \overline{I}$. This module is torsion-free by the previous lemma, so $\tau(M_i) = \varinjlim \underline{\mathrm{Hom}}(A / \m^n, M_i) = 0$. In particular $\underline{\mathrm{Hom}}(A / \m, M_i) = 0$, and so by \cite[Lemma 8.8]{chan_species_2016} we have the projective dimension $\mathrm{pd}~ M_i \neq 2$. The noncommutative symmetric algebra has global dimension 2 \cite[Proposition 6.1]{chan_species_2016}, so $\mathrm{pd}~ M_i \leq 1$, and hence $\mathrm{pd}~ e_i\overline{I} \leq 0$. It follows that $e_i\overline{I}$ is free, so there is some collection $\{ i_1, \ldots i_r \}$ of integers for which $e_i\overline{I} \cong \bigoplus_{j=1}^r e_{i_{j}}A$. Note all the $i_j$ must be greater than $i$ due to positive indexing. It follows from Proposition 3.10 that $f_{e_iA}(n) = n+1$, so we see that for $n\gg 0$ the Hilbert function of $e_i\overline{I}$ is 
\[f_{e_i\overline{I}}(n) = \sum_{j=1}^r n-(i_j-i) + 1 = rn + (r(i+1)-\sum_{j=1}^r i_j).\]
Thus the multiplicity of $e_i\overline{I}$ is $r$. As $e_iA$ is critical \cite[Theorem 6.2(1)]{chan_species_2016} and has multiplicity 1, we must have $r=1$. Hence $e_i\overline{I}$ is free of rank 1.
\end{proof}

The next result now follows immediately by applying Theorem 4.7. 
\begin{cor}
The ideal $\overline{I}$ of $A$ is generated by a normal family of elements $\{a_i\}$.
\end{cor}

We now recall the definition of a $\Z$-indexed \textit{twisted ring}, as introduced by Chan and Nyman in \cite{chan_species_2016}.

\begin{defn}
    Let $L$ be a field, and let $\psi = \{ \psi_{i}\}_{i\in \Z}$ be a collection of $L$-automorphisms. The \textit{full twisted ring} $C = (L;\psi)$ is the $\Z$-indexed algebra defined by $C_{ij} = L$ for all $i,j \in \Z$, with multiplication defined by
    \[c_{ij}c_{jl} = c_{ij}\cdot \psi_{i+1}\psi_{i+2}\ldots\psi_{j}(c_{jl}) \in C_{il},\]
    where $c_{ij} \in C_{ij}$, $c_{jl} \in C_{jl}$, and $\cdot$ is the usual $L$-multiplication. A \textit{twisted ring} on $L$ is any subring $C' \subset C$ such that $C'_{\geq d} = C_{\geq d}$ for some $d \geq 0$.
\end{defn}

\begin{remark}
    If we let $\bimod{i}{\psi}{j}:= \psi_{i+1} \circ \ldots \circ \psi_j$, then we may reinterpret the notion of a full twisted ring in the notation from Section 3. We let $C_{ij}:= L_{\bimod{i}{\psi}{j}}$, and the multiplication $C_{ij} \times C_{jl} \to C_{il}$ is the twisted $L$-multiplication
    \[c_{ij}c_{jl} = c_{ij} \cdot \bimod{i}{\psi}{j}(c_{jl}),\]
    where $\cdot$ denotes the usual $L$-multiplication. This multiplication is well-defined as $\bimod{i}{\psi}{j} \bimod{j}{\psi}{l} = \bimod{i}{\psi}{l}$.
\end{remark}

\begin{lemma}
    Let $L$ be a field and $C = (L,\{\psi_i\}_{i\in\Z})$ be a twisted ring on $L$. If $J$ is an ideal with $e_iJ \neq 0$ for all $i \in \Z$, then any $C/J$-module is $\m_C$-torsion.
\end{lemma}
\begin{proof}
    The right modules $e_iC$ have Gelfand-Kirillov dimension 1, as $C$ is a full twisted ring in high degrees. As $C$ is a domain and each $e_iJ \neq 0$, we have
    \[\dim e_i (C / J) \leq \dim e_iC - 1 = 0,\]
    so each $e_i(C/J)$ is $\m_C$-torsion. Then any $C/J$-module must be $\m_C$-torsion.
\end{proof}

Recall from \cite{chan_species_2016} that the quotient $\Z$-indexed algebra $B = A/gA$ is a twisted ring on $F$, full in positive degrees, so in particular the previous proposition applies to $B$. For the duration of the next proposition we refer to the ideal $(g)$ by $\mathfrak{g}$.

\begin{propn}
If $I$ is idempotent, $\overline{I} / \overline{I}^2$ is non-zero and has projective dimension 2.
\end{propn}
\begin{proof}
    We note that $\overline{I}$ cannot be idempotent, as $A$ is connected and positively indexed, and $\overline{I}$ is contained in $\m$. As such, $\overline{I} / \overline{I}^2$ is non-zero. Using Lemma 4.8 we see 
    \[(\overline{I} / \overline{I}^2)[g^{-1}] = \widetilde{I} / \widetilde{I}^2 = 0\] 
    as $\widetilde{I}$ is idempotent, so $\overline{I} / \overline{I}^2$ is $g$-torsion. Then the annihilator $J := \mathrm{Ann}_{\overline{I} / \overline{I}^2}g$ of $g$ in $\overline{I} / \overline{I}^2$ must be nonzero. By construction $J$ is $\overline{I}$-torsion, so $J$ is an $A / (\mathfrak{g} + \overline{I})$-module. Two-sided ideals in $\Z$-indexed algebras are subobjects in the abelian category $\mathsf{Bimod}(A)$, so we see
    \[A / (\mathfrak{g} + \overline{I}) \cong (A / \mathfrak{g}) / ((\mathfrak{g} + \overline{I})/\mathfrak{g})=B / \overline{I}B,\] 
    as $(\mathfrak{g} + \overline{I})/\mathfrak{g}$ is the image of the ideal $\overline{I}$ in $B$. The previous lemma then tells us that $J$, and hence also $\overline{I} / \overline{I}^2$, must contain $\m$-torsion elements. Finally, \cite[Lemma 8.8]{chan_species_2016} then says that $\overline{I} / \overline{I}^2$ must have projective dimension 2.
\end{proof}

\begin{theorem}
    $\Lambda_{00}$ is a Dedekind domain.
\end{theorem}
\begin{proof}
    It is shown in \cite[Theorems 5.2, 6.2]{chan_species_2016} that $A$ is noetherian and a domain, so it is clear that $\Lambda_{00}$ also inherits these properties. 

    For any $\Lambda$-module $M$, there is a $g$-torsion-free $A$-module $\overline{M}$ such that $\overline{M}[g^{-1}] = M$. Similarly, for any projective resolution $P^{\bullet}$ of $M$ there is a projective resolution $\overline{P}^{\bullet}$ of $\overline{M}$ such that the following diagram (where all the vertical maps are $(-)[g^{-1}]$) commutes.

    \begin{center}
        \begin{tikzcd}
        \overline{P}^{2} \arrow[r] \arrow[d] & \overline{P}^{1} \arrow[r] \arrow[d] & \overline{P}^{0} \arrow[r] \arrow[d] & \overline{M} \arrow[d] \\
        P^{2} \arrow[r] & P^{1} \arrow[r] & P^{0} \arrow[r] & M.
    \end{tikzcd}
    \end{center}

    Applying the quotient functor $\pi$ gives us projective resolutions $\pi \overline{P}^{\bullet}$ and $\pi P^{\bullet}$ of the coherent sheaves $\pi \overline{M}$ and $\pi M$ on the schemes $\PP^{nc}(F)$ and $\qgr \Lambda$, respectively. The diagram above becomes the following, where now the vertical maps are all $\pi((-)[g^{-1}])$.
    
    \begin{center}
        \begin{tikzcd}
        \pi\overline{P}^{2} \arrow[r] \arrow[d] & \pi\overline{P}^{1} \arrow[r] \arrow[d] & \pi\overline{P}^{0} \arrow[r] \arrow[d] & \pi\overline{M} \arrow[d,] \\
        \pi P^{2} \arrow[r] & \pi P^{1} \arrow[r] & \pi P^{0} \arrow[r] & \pi M.
    \end{tikzcd}
    \end{center}
    
    Now $\PP^{nc}(F)$ is hereditary \cite[Corollary 8.7]{chan_species_2016}, hence $\qgr \Lambda$ must also be. Then both $\Lambda$ and $\Lambda_{00}$ must be hereditary, as $\qgr \Lambda =\gr \Lambda$.

    Finally, we show that $\Lambda_{00}$ has no nontrivial idempotent ideals. For any proper ideal $I$ of $\Lambda_{00}$, the ideals $\overline{I}$ and $\overline{I}^2$ of $A$ are principal. It is clear that $\overline{I}$ is principal, as $\overline{I}$ is generated by a normal family of elements $\{a_i\}_{i \in \Z}$ of degree $\delta$ by Corollary 4.11. To see that $\overline{I}^2$ is principal, note that any element of $\overline{I}^2_{il}$ can be written as a sum of products $a_ixa_jy$ with $x \in A_{i+\delta, j}$ and $y \in A_{j+\delta, l}$. Then, by normality, there is some $x' \in A_{i+2\delta, j}$ such that $a_ixa_jy = a_ia_{i+\delta}x'y$. So $\overline{I}^2$ is the ideal generated by the normal family $\{a_ia_{i+\delta}\}$. Then 
    \[0 \to \overline{I}^2 \to \overline{I} \to \overline{I} / \overline{I}^2 \to 0\] 
    is a projective resolution of $\overline{I} / \overline{I}^2$, and thus $\mathrm{pd}~\overline{I} / \overline{I}^2 \leq 1$. But Proposition 4.15 says $\mathrm{pd}~\overline{I} / \overline{I}^2 =2$ if $I$ is idempotent, and so $I$ cannot be idempotent.
\end{proof}

\subsection{$\Lambda_{00}$ has dimension 1}
We end this section by showing that $\Lambda_{00}$ has dimension 1, which together with Theorem 4.16 shows $\Lambda_{00}$ should be thought of as a noncommutative affine curve. The $g$-adic filtration on $A$ induces a (now ascending) filtration $\Lambda_{00}^i$ on $\Lambda_{00}$ in the following way. The $i$-th piece of the filtration $\Lambda_{00}^i$ is $A_{0,2i}g^{-i}$, with inclusion maps 
\begin{align*}
    \iota_{ij}:\Lambda_{00}^i &\hookrightarrow \Lambda_{00}^{j} \\
    ag^{-i} &\mapsto (ag^{j-i})g^{-j}
\end{align*}
for $i \leq j$. Furthermore, $\Lambda_{00}^i \Lambda_{00}^1 = \Lambda_{00}^1 \Lambda_{00}^i = \Lambda_{00}^{i+1}$ as $A$ is generated in degree 1 and $g$ is a normal family.

\begin{propn}
    The Hilbert function of $\Lambda_{00}$ is $f_{\Lambda_{00}}(n) = \begin{cases}
        1 & \text{ if }n=0, \\
        2 & \text{ if }n>0,
    \end{cases}$ and hence $\dim\Lambda_{00} = 1$.
\end{propn}
\begin{proof}
    It suffices to show the claims hold for the associated graded ring $R=\mathrm{gr}_g(\Lambda_{00})$ by \cite[Proposition 6.6]{krause2000growth}. Via Theorem 3.12 we see that $R_0 = K_0$ and $R_i \cong F_{\tau_{0,2i}}$ for $i>0$. Then $f_R(0)=1$ and $f_R(n) = 2$ if $n>0$, so $\dim R = \deg f_R(n)+1= 1$. 
\end{proof}

\section{Algebraicity}
To further study the ideals of $\Lambda_{00}$ we will need to treat two distinct cases. Following \cite{ringel_representations_1976}, we say a two-sided vector space $\bimod{K_0}{V}{K_1}$ is \textit{algebraic} if there is a common subfield $k \subset K_0, K_1$ of finite index which acts centrally on $V$. Otherwise, $V$ is \textit{non-algebraic}. 

\begin{ex}
    Consider the field $L = \Q(i, \sqrt[4]{2})$. Then $G = \mathrm{Gal}(L/\Q) \cong D_{4}$, generated by the `reflection'
    \[\tau:i \mapsto -i\]
    and the `rotation'
    \[r: \sqrt[4]{2} \mapsto i\sqrt[4]{2}.\]
    So to turn $L$ into an algebraic two-sided vector space of rank 2 is to choose two order 2 subgroups $G_0,G_1 \leqslant G$, as the fixed fields $L^{G_i}$ have index 2 in $L$. For example, let $G_0 = \langle \tau \rangle$ and $G_1 = \langle r^2\tau \rangle$. Then $L^{G_0} = \Q(\sqrt[4]{2})$ and $L^{G_1} = \Q(i\sqrt[4]{2})$, so $L^{G_0} \cap L^{G_1} = L^{\langle\tau,r^2\rangle}=\Q(\sqrt{2})$. 
\end{ex}

\begin{ex}
    The transcendence degree 1 extension $\C(t)$ of $\C$ has index 2 subfields $\C(t^2)$ and $\C((t-1)^2)$. Then $\C(t)$ is a rank 2, two-sided $\C(t^2)-\C((t-1)^2)$-vector space. The two Galois involutions are
    \[\tau_0: t \mapsto -t\]
    and
    \[\tau_1: t-1 \mapsto 1-t,\]
    and hence their composite is
    \[\tau_1\tau_0: t \mapsto t - 2.\]
    This automorphism has infinite order, as it will never send $t$ to a multiple of $t-1$, and so $k = \C$ and $\bimod{\C(t^2)}{\C(t)}{\C((t-1)^2)}$ is non-algebraic. 
\end{ex}

We remark that Example 2.5 also contains an example of a non-simple, non-algebraic two-sided vector space. For the remainder of the paper we will assume $V=F$ is a simple two-sided vector space of rank 2. The main goal of Section 5 is to show algebraicity determines the dichotomy on $\Lambda_{00}$ presented in Theorem 1.4, relating $\Lambda_{00}$ to the two-sided vector space $\bimod{K_0}{F}{K_1}$ and to the group $H = \langle \tau_0,\tau_1 \rangle$.

\begin{theorem}
Let $\sigma := \tau_1\tau_0 \in H$ be the composition of the Galois involutions. Then
    \begin{itemize}
    \item $\bimod{K_0}{F}{K_1}$ is algebraic $\Longleftrightarrow$ $\vert \sigma \vert < \infty$ $\Longrightarrow$ $\Lambda_{00}$ is finite over its center,
    \item $\bimod{K_0}{F}{K_1}$ is non-algebraic $\Longleftrightarrow$ $\vert \sigma \vert = \infty$ $\Longrightarrow$ $\Lambda_{00}$ is a simple ring.
\end{itemize}
\end{theorem}

The proof of this result is contained in Corollary 5.5 and Theorems 5.6 and 5.9. Theorem 5.3 resembles Theorem 7.3 of \cite{artin_modules_1991}, which applies to a similar affine open subset of a noncommutative quadric (divisor of bi-degree (2,2) in $\PP^1 \times \PP^1$, or in other words a rank 2 $\mathcal{O}_{\PP^1}$-bimodule). In a sense, Theorem 5.3 can be viewed as a version of \cite[Theorem 7.3]{artin_modules_1991} `over the generic point' of a $\Z$-indexed noncommutative quadric, as introduced by Van den Bergh in \cite[Section 5]{van_den_bergh_noncommutative_2011}.

\subsection{Algebraicity and the group $H$}
For a simple two-sided vector space $\bimod{K_0}{F}{K_1}$, any subfield $L$ of $K_0$ and $K_1$ must be contained in the intersection $K_0 \cap_F K_1$. We may now interpret algebraicity of $\bimod{K_0}{F}{K_1}$ in terms of the group $H$ as follows.

\begin{propn}
    $\bimod{K_0}{F}{K_1}$ is algebraic if and only if $H$ is a finite group.
\end{propn}
\begin{proof}
      The discussion above shows that if $F$ is algebraic, the finite index subfield $k$ of $K_0$ and $K_1$ acting centrally on $F$ must be contained in the intersection $K_0 \cap_F K_1$, and hence $K_0 \cap_F K_1$ must have finite index in $K_0$ and $K_1$. The result then follows from the fundamental theorem of Galois theory \cite[Chapter V, \S10.7. Theorem 4]{bourbaki_algebra_2013}.
\end{proof}

We then immediately obtain the first part of Theorem 5.3.

\begin{cor}
    $\bimod{K_0}{F}{K_1}$ is algebraic if and only if $\vert \sigma \vert < \infty$.
\end{cor}

\subsection{{$\Lambda_{00}$} for algebraic two-sided vector spaces}
We will now fix $\bimod{K_0}{F}{K_1}$ to be an algebraic two-sided vector space, with central subfield $k$ such that $[F:k] = 2m$. The torsion sheaves on $\PP^{nc}(F)$ are well understood in this case, as it follows from \cite[Theorem 1]{ringel_representations_1976} that the category $\fl \Lambda_{00}$ of finite-length $\Lambda_{00}$-modules is a direct sum of uniserial categories of global dimension 1 with a unique simple object. As $\Lambda_{00}$ is a Dedekind domain every torsion module is a direct sum of cyclic modules, and hence each of these uniserial categories consists of self-extensions of the simple cyclic $\Lambda_{00}$-modules. We may now complete the first half of the dichotomy presented by Theorem 5.3.

\begin{theorem}
    If $\bimod{K_0}{F}{K_1}$ is algebraic, $\Lambda_{00}$ is finite over its center $Z(\Lambda_{00})$.
\end{theorem}
\begin{proof}
    The intersection of $K_0$ and $K_1$ in $F$ is fixed by both $\tau_0$ and $\tau_1$, so as $k \subset K_0 \cap_F K_1$ we have $k \subset Z(\Lambda_{00})$. Recall that the $g$-adic filtration on $\Lambda_{00}$ is given by $\Lambda_{00}^i = A_{0,2i}g^{-i}$. As $\dim_k K_0$ and $\dim_{K_0}A_{0,2i}$ are both finite, $\dim_k \Lambda_{00}^i$ is also finite. Then, as $\Lambda_{00}^i = (\Lambda_{00}^1)^i$, $\Lambda_{00}$ is finitely generated as a $k$-algebra by $\Lambda_{00}^1$. The Gelfand-Kirillov dimension of $\Lambda_{00}$ is $1$ by Proposition 4.17, so the Small-Stafford-Warfield theorem \cite{small1985affine} tells us that the quotient by the nilradical $\Lambda_{00} / N(\Lambda_{00})$ is a finite module over its center. The result then follows as $\Lambda_{00}$ is a domain, so $N(\Lambda_{00})$ is trivial.
\end{proof}

\subsection{$\Lambda_{00}$ for non-algebraic two-sided vector spaces}
We now suppose $\bimod{K_0}{F}{K_1}$ is a non-algebraic two-sided vector space. The group $H$ is always abstractly isomorphic to a dihedral group, as it is generated by two involutions. When $\bimod{K_0}{F}{K_1}$ is non-algebraic $H$ must be abstractly isomorphic to the infinite dihedral group $\mathbb{D}_{\infty}$, as the composition of the two generating involutions has infinite order via Corollary 5.5. For $i \in \{0,1\}$ we let $\sigma_i$ denote $\tau_i\tau_{1-i} \in H$, in which case we may express $H$ as a semidirect product $H \cong \langle \sigma_i \rangle \rtimes G_i$.

\begin{lemma}
    If $u \in H$ and $u(K_i) = K_j$ for some $i,j \in \{0,1\}$, then $i=j$ and $u \in G_i=\{1,\tau_i\}$.
\end{lemma}
\begin{proof}
    Suppose $u(K_i)=K_j$. Then we have 
    \[u(K_i) = u(F^{G_i}) = F^{uG_iu^{-1}} = F^{G_j}=K_j.\] 
    Then $\tau_i$ and $\tau_j$ are conjugate as $uG_iu^{-1} = \{ 1, u\tau_iu^{-1}\} = G_j$, and so we must have $i=j$ \cite[Chapter IV, \S1.3. Proposition 3]{bourbaki_lie_2002}. Thus $u \tau_i = \tau_iu$. Using the fact that $H \cong \langle \sigma_i \rangle \rtimes G_i$, we express $u$ as a product $u = \sigma_i^m\tau_i^r$ with $m \in \Z$ and $r \in \{0,1\}$. Then
    \[u\tau_i = \sigma_i^m\tau_i^{r+1} \]
    and
    \[ \tau_i u = \tau_i\sigma_i^m\tau_i^r=\sigma_i^{-m}\tau_i^{r+1}, \]
    so we must have $m=0$, as $\vert \sigma_i \vert = \infty$ by Corollary 5.5. It follows that $u \in \{1,\tau_i\} = G_i$.
\end{proof}

Given a normal family $x=\{x_i\}_{i \in \Z}$ in $A$, we again let $(x)$ denote the two-sided ideal generated by the $x_i$.

\begin{propn}
    If $x=\{x_i\}_{i \in \Z}$ is a normal family of degree $\delta$ in $A$, then $\delta$ is even and $(x) = (g^{\delta/2})$. 
\end{propn}
\begin{proof}
    Recall from Section 2.1 that for $d \in \Z$, $A[d]$ is the `shifted' $\Z$-indexed algebra defined by $A[d]_{ij} = A_{i+d,j+d}$ for all $i,j \in \Z$. As $A$ is a domain \cite[Theorem 6.2(2)]{chan_species_2016} each $x_i$ is a non-zero-divisor, and hence the normal family $x$ determines a `conjugation' automorphism $\varphi_x:A \to A[\delta]$ where, for $a \in A_{ij}$, $\varphi_x(a)\in A_{i+\delta,j+\delta}$ is such that $ax_j = x_i \varphi_x(a)$. This means $\varphi_x:K_i=A_{ii} \to A_{i+\delta,i+\delta}=K_{i+\delta}$ is an isomorphism, and so $\delta$ must be even by Lemma 5.7. 
    
    Using Theorem 3.12 we decompose each $x_i$ as a sum 
    \[x_i = \displaystyle\sum_{m=0}^{\delta/2} x_{i,m} \in K_i \oplus  F_{\tau_{i,i+2}} \oplus F_{\tau_{i,i+4}} \oplus \ldots \oplus F_{\tau_{i,i+\delta}}, \]
    where $\tau_{i,i+\delta} = (\tau_{i+1},\tau_{i+2},\ldots,\tau_{i+\delta-1}) \in G_{i,i+\delta}$. Note that as $H \cong \mathbb{D}_{\infty}$, $\overline{\tau_{ij}}(K_{i}) \neq \overline{\tau_{i,j+2l}}(K_{i})$ unless $l=0$. For $b \in K_i$, there is some $c=\varphi_x(b) \in K_{i+\delta}$ such that $bx_i=x_ic$. But
    \[ x_ic = \sum_{m=0}^{\delta/2} x_{i,m} \cdot \overline{\tau_{i,i+2m}}(c) = \sum_{m=0}^{\delta/2} b\cdot x_{i,m}, \]
    where $\cdot$ is the usual $F$-multiplication, so $x_{i,m}$ is non-zero for at most one $m$. Fix $0\leq l \leq \delta/2$ such that $x_i=x_{i,l}$, in which case $c = \overline{\tau_{i,i+2l}}^{-1}(b)$. Then $\overline{\tau_{i,i+2l}}^{-1}$ defines an isomorphism $K_i \xrightarrow{\sim} K_{i+\delta}$, and so we must have $l=0$ by Lemma 5.7. Thus $x_i=x_{i,0} \in K_i \subset A_{i,i+\delta}$, and this 1-dimensional $K_i$-vector space contains $g_i^{\delta / 2}=g_{i}g_{i+2}\ldots g_{i+\delta-2}$, so $x_i = ag_i^{\delta / 2}$ for some $a \in K_i$. Thus $(x) = (g^{\delta/2})$.
\end{proof}

With this, we can complete the proof of Theorem 5.3.

\begin{theorem}
    If $\bimod{K_0}{F}{K_1}$ is a simple, non-algebraic two-sided vector space, $\Lambda_{00}$ is a simple ring.
\end{theorem}
\begin{proof}
    By Corollary 4.11, any two-sided ideal $I\neq 0$ of $\Lambda_{00}$ gives a two-sided ideal $\overline{I}$ of $A$ generated by a family of normal elements. The previous proposition shows that $\overline{I} = (g^n)$ for some $n\geq 1$, and so we must have that $I = \overline{I}[g^{-1}]_{00} = \Lambda_{00}$ by Lemma 4.8.
\end{proof}

We remark that when $\bimod{K_0}{F}{K_1}$ is non-algebraic, $\Lambda_{00}$ appears to be a new example of a simple noncommutative Dedekind domain. It is stated shortly after Theorem 1.1 of \cite{smertnig2017every} that simple noncommutative Dedekind domains are obtained as skew polynomial and Laurent rings, of which it seems $\Lambda_{00}$ is neither. Indeed, $\Lambda_{00}$ is not generated over $K_0$ by a single element nor by units, and hence $\Lambda_{00}$ is not isomorphic to $K_0[t;\psi,\delta]$ nor $K_0[t,t^{-1};\psi]$ for any $K_0$-automorphism $\psi$ and $\psi$-derivation $\delta$.


\end{document}